\documentclass[11pt,a4paper]{amsart}

\usepackage{amssymb,amsmath,graphics,verbatim}
\usepackage{latexsym}
\usepackage{eucal}
\usepackage{a4wide}

\newtheorem{proposition}{Proposition}[section]
\newtheorem{theorem}{Theorem}[section]
\newtheorem{lemma}[theorem]{Lemma}

\newtheorem{coro}[theorem]{Corollary}
\newtheorem{remark}[theorem]{Remark}

\newcommand{\mc}{\mathcal}

\newcommand{\rr}{\mathbb{R}}
\newcommand{\R}{\mathbb{R}}
\newcommand{\nn}{\mathbb{N}}
\newcommand{\cc}{\mathbb{C}}

\newcommand{\zz}{\mathbb{Z}}
\newcommand{\qq}{\mathbb{Q}}

\newcommand{\eps}{\epsilon}

\newcommand{\pl}{\partial}
\newcommand{\x}{\times}

\newcommand{\til}{\widetilde}
\newcommand{\bbar}{\overline}

\newcommand{\cjd}{\rangle}
\newcommand{\cjg}{\langle}

\newcommand{\demi}{\frac{1}{2}}

\newcommand{\indic}{\operatorname{1\negthinspace l}}

\def\qed{\hfill$\square$\medskip}

\begin{document}
\title[Identification of a connection from Cauchy data]{Identification of a connection from Cauchy data on a Riemann surface with boundary}
\author{Colin Guillarmou}
\address{DMA, U.M.R. 8553 CNRS\\
Ecole Normale Sup\'erieure,\\
45 rue d'Ulm\\ 
F 75230 Paris cedex 05 \\France}
\email{cguillar@dma.ens.fr}
\author{Leo Tzou}
\address{Department of Mathematics\\
Stanford University\\
Stanford, CA 94305, USA, and  Department of Mathematics \\
University of Helsinki \\
PO Box 68, 00014 Helsinki, Finland}
\email{ltzou@gmail.com}

\begin{abstract}
We consider a connection $\nabla^X$ on a complex line bundle over a Riemann surface with boundary $M_0$, 
with connection 1-form $X$.  We show that the Cauchy data space of the connection Laplacian (also called 
magnetic Laplacian) $L:={\nabla^X}^*\nabla^X + q$, with $q$ a complex valued potential,
uniquely determines the connection up to gauge isomorphism, and the potential $q$.
\end{abstract}

\maketitle

\begin{section}{Introduction}
Let $M_0$ be a smooth Riemann surface with boundary, equipped with a metric $g$.
A complex line bundle $E$ on $M_0$ has a trivialization $E\simeq M_0\x\cc$, thus there is a non-vanishing smooth section
$s:M_0\to E$,  and  a connection $\nabla$ on $E$ induces  
a complex valued $1$-form $iX$ on $M_0$ (where $i=\sqrt{-1}\in\cc$) defined by
$\nabla s=s\otimes iX$, which means that $\nabla (fs)=s\otimes (d+ iX)f$ if $d$ is the exterior derivative. 
The associated connection Laplacian ($*$ is the Hodge operator with respect to $g$) is the operator
\[\Delta^X := {\nabla^X}^*\nabla^X = -*(d* +i X\wedge *)(d +iX)\] 
acting on complex valued functions (sections of $E$). 
When $X$ is real valued, this operator is often called the \emph{magnetic Laplacian} associated 
to the magnetic field $dX$, and the connection $1$-form $X$ can be seen as to a connection $1$-form
on the principal bundle $M_0\x S^1$ by identifying $i\rr\subset \cc$ with the Lie algebra of $S^1$. This also corresponds 
to a Hermitian connection, in the sense that it preserves the natural Hermitian product on $E$.
Let $q$ be a complex valued  function on $M_0$ and assume that the $1$-form $X$ is real valued,
and consider the \emph{magnetic Schr\"odinger Laplacian}
associated to the couple $(X,q)$ 
\begin{equation}\label{defL}
L :=  {\nabla^X}^*\nabla^X + q= -*(d* + iX\wedge *)(d +iX)+q.
\end{equation}
If $H^s(M_0)$ denotes the  Sobolev space with $s$ derivatives in $L^2$,
we define the Cauchy data space of $L$ to be
\begin{equation}\label{CL}
\mc{C}_L := \{ (u,\nabla^X_{\nu} u)|_{\pl M_0}\in H^{\demi}(M_0)\x H^{-\demi}(M_0);  u\in H^1(M_0), Lu=0\}
\end{equation}
where $\nu$ is the outward pointing unit normal vector field to $\pl M_0$ and 
$\nabla_\nu^Xu:=(\nabla^Xu)(\nu)$ . 
The first natural inverse problem is to see if the Cauchy data space determines the connection form $X$ and the potential $q$ uniquely,
and one easily sees that it is not the case since there are gauge invariances in the problem: for instance,
conjugating $L$ by $e^{f}$ with $f=0$ on $\pl M_0$, one obtains the same Cauchy data space but with a Laplacian associated to the  connection $\nabla^{X+df}$, therefore it is not possible to identify $X$ but rather one should expect to recover its relative cohomology class.

In general in inverse problems for magnetic Laplacians, it is shown that if two couples $(X_1,q_1)$ and $(X_2,q_2)$ are such that the associated connection Laplacian 
$L_1$ and $L_2$ have same Cauchy data space, then $d(X_1-X_2)=0$ and $q_1=q_2$. In geometric terms, this means here that the curvature of the connections $\nabla^{X_1}$ and $\nabla^{X_2}$ agree. If the domain is simply connected, $X_1$ would then differ from $X_2$ by an exact form, and moreover $i_{\pl M_0}^*(X_1-X_2)$=0 by boundary determination . For these types of results in Euclidean domains of dimensions three and higher, we refer the readers to the works of Henkin-Novikov \cite{HeNo2}, Sun \cite{Sun1, Sun2}, Nakamura-Sun-Uhlmann in \cite{NaSuUh}, 
Kang-Uhlmann in \cite{KaUh}, and for partial data Dos Santos Ferreira-Kenig-Sj\"ostrand-Uhlmann in \cite{dksu}. 
%For simply connected manifolds of dimensions three and higher, see Dos Santos Ferreira-Kenig-Salo-Uhlmann in \cite{dksu2}. 
For simply connected planar domains, Imanuvilov-Yamamoto-Uhlmann in \cite{IUY2} deal with  
the case of general second order elliptic operators for partial data measurement, and Lai \cite{Lai} deals with the special case
of magnetic Schr\"odinger operator for full data measurement. 

However, on a general Riemann surface with boundary, the first cohomology space is non-trivial in general and 
$X_1 - X_2$ may not be exact and there is another gauge invariance. Indeed, if $X_j$ are real valued and 
if there exists a unitary bundle isomorphism $F:E\to E$ (i.e. preserving the Hermitian product) such that $\nabla^{X_1}=F^{*}\nabla^{X_2}F$ with $F={\rm Id}$ on $E|_{\pl M_0}$, then it is clear that the Cauchy data spaces $\mc{C}_{L_1}=\mc{C}_{L_2}$ 
agree. We shall say in this case that the connections are \emph{related by a gauge isomorphism}.
Such a bundle isomorphism corresponds to the multiplication by a function $F$ on $M_0$ satisfying 
$|F|=1$ everywhere and $F=1$ on $\pl M_0$ and this is equivalent to have $i_{\pl M_0}^*(X_1-X_2)=0$ on $\pl M_0$,
$d(X_1-X_2)=0$ and that $\int_{\gamma} (X_1-X_2)\in 2\pi \zz$ for all closed loop $\gamma$ in $M_0$. 
Another way of stating this isomorphism is the following: let $\gamma_1,\dots,\gamma_M$ be some non-homotopically 
equivalent closed loops of $M_0$, non-homotopically equivalent to any boundary component, and let
$\omega_1,\dots, \omega_M$ be closed $1$-forms which form 
a basis of the first relative cohomology group $H^1(M_0,\pl M_0)$, dual to $\gamma_1,\dots,\gamma_M$ 
in the sense $\int_{\gamma_k}\omega_j=\delta_{ij}$, then there is a bundle isomorphism as above if and only if
$X_1=X_2+2\pi \sum_{m=1}^Mn_m\omega_m$ with $n_m\in\zz$. \\ 
%The holonomy group $H_m(\nabla^X)$ of the connection $\nabla^X$ 
%at a point $m\in M_0$ correspond to the subgroup of $GL(\cc)=\cc\setminus \{0\}$ induced by parallel transports $P_\gamma$ 
%along oriented closed loops $\gamma$ based at $m$
%\[H_{m}(\nabla^X):=\{P^X_\gamma\in \cc\setminus\{0\}; \gamma \textrm{ is a closed loop based at }m \}.\]
%see Section \ref{holonomy} for more details.  
%If $X$ is real valued, this is a subgroup of  $S^1$ generated by 
%$e^{i\int_\gamma X}$ for oriented closed loops $\gamma$ based at $m$. 
%If $\pi_1(M_0,m)$ denotes the fundamental group based at $m$ (ie. the closed loops based at $m$ up to homotopy equivalence) 
%and if $X$ is a flat connection $1$-form, ie. $dX=0$, then 
%the map  $\gamma \to P_\gamma^{X}$  induces a surjective group morphism 
%$\rho^X_m: \pi_1(M_0,m)\to H_m(\nabla^X)\subset S^1$, we call it the 
%\emph{holonomy representation} of $\nabla^X$, it is a representation of the fundamental group 
%into the group $GL(\cc)$. The holonomy representation is trivial when $H_m(\nabla^X)=\{1\}$, which is equivalent to say that $\int_\gamma X\in 2\pi\zz$ for all closed loop $\gamma$, this condition is independent of $m$.\\

For $s\in \nn,p\in[1,\infty]$, let us denote by $W^{s,p}(M_0)$ the Sobolev space consisting of functions with $s$ derivatives in $L^p$.
The following theorem provides a characterization of precisely when two Cauchy data spaces agree:
\begin{theorem} \label{main}
Let $X_1,X_2\in W^{2,p}(M_0,T^*M_0)$ be real valued $1$-forms and let $q_1,q_2\in W^{1,p}(M_0)$ be complex valued functions,  
where $p>2$. 
Let $L_1,L_2$ be the magnetic Schr\"odinger Laplacians 
defined in \eqref{defL} for the couples $(X_1,q_1)$ and $(X_2,q_2)$. 
Then the Cauchy data spaces $\mc{C}_{L_1}$ and $\mc{C}_{L_2}$ coincide if and only if 
$q_1=q_2$, and $\nabla^{X_1}$ is related to $\nabla^{X_2}$ by a gauge isomorphism.
\end{theorem}

As far as we know, this is the first result with such a characterization in the case of elliptic equations at fixed frequency. 
There are results for time-dependent inverse problems, or equivalently when one knows the Cauchy data spaces at all frequencies, for instance by Kurylev-Lassas \cite{KuLa} and Eskin-Isozaki-O'Dell \cite{EsIs} 
(see the references therein for results about inverse scattering). 
The problem of recovering the fluxes $\int_\gamma X$ (modulo $2\pi \zz$) of the magnetic potential along closed loops 
is related to the so-called Aharonov-Bohm effect \cite{AhBo}. We notice that for the free case $X=0$, the identification of $q$ (or of an isotropic  conductivity) on a Riemann surface with boundary from full data measurement was proved recently in \cite{HeMi,GT}, 
and in \cite{GT2} from partial data measurement, it was done in a domain of $\cc$ in \cite{IUY}.\\

For a general complex valued connection form $X$ and the associated operator $L$ of \eqref{defL}, 
we can define its real and imaginary parts of $iX$ by
$X_\rr:=\demi (iX+\bbar{iX})$ and $X_{i\rr}:=\demi(iX-\bbar{iX})$ and by an easy computation we have
\[{\nabla^{X}}^*\nabla^X={\nabla^{X_{i\rr}}}^*\nabla^{X_{i\rr}}+\delta(X_{\rr})+|X_{\rr}|_g^2\] 
with $\delta:=d^*=-* d*$ 
and therefore, comparing to the case where $X$ was real, 
it changes the operator $L$ in \eqref{defL} only through the addition of a potential, namely 
$\delta(X_\rr)+|X_\rr|^2$. Considering the Cauchy data space of $L$ to be 
\[\mc{C}_L := \{ (u,\nabla^{X_{i\rr}}_{\nu} u)|_{\pl M_0}\in H^{\demi}(M_0)\x H^{-\demi}(M_0);  u\in H^1(M_0), Lu=0\}\] 
we see from Theorem \ref{main} that the Cauchy data space $\mc{C}_L$ determines 
$\delta(X_\rr)+|X_\rr|^2+q$, and the imaginary part of $iX$ up to gauge isomorphism.\\

To prove the identification of the potential and the curvature of the connection, we shall reduce the problem to a first order system and we shall actually prove an inverse result for the following auxiliary problem:
let us define the bundle $\Sigma:= \Lambda^{0}(M_0)\oplus \Lambda^{0,1}(M_0)$  over $M_0$ where $\Lambda^k(M_0)$
denotes the bundle of complex valued $k$-forms on $M_0$ (for $k=0,1,2$) and 
\[\Lambda^{0,1}(M_0)=T^*_{0,1}M_0=\ker_{\Lambda^1}(*-i{\rm Id}), \,\,\Lambda^{1,0}(M_0)=T^*_{0,1}M_0
=\ker_{\Lambda^1}(*+i{\rm Id})\]
and $*$ is the Hodge star operator on $\Lambda^k(M_0)$.
Let $D := C^\infty(\Sigma)\to C^\infty(\Sigma)$ be the self adjoint operator defined by 
 \[D := \begin{pmatrix} {0}&{\bar{\pl}^*}\cr{\bar\pl}&{0}\end{pmatrix}\] 
and $V$ be a complex valued diagonal endomorphism, then we set $P:=D+V$. 
The Cauchy data space is defined by 
\[\mc{C}_V:=\{U|_{\pl M_0}; U\in W^{1,p}(M_0,\Sigma), (D+V)U=0\}.\]
and we prove
\begin{theorem}\label{th2}
\label{first order inverse problem}
Let $V_1,V_2\in W^{1,p}(M_0,{\rm End}(\Sigma))$ be two diagonal endomorphisms of $\Sigma$. Assume that they have the same Cauchy data spaces, ie.  $\mc{C}_{V_1}=\mc{C}_{V_2}$, then $V_1=V_2$. 
\end{theorem}
Such a result was recently proved by Bukhgeim \cite{Bu} in the case of the disk in $\cc$, and we use a similar approach 
to solve this problem in this geometric setting, together with some results on holomorphic Morse functions proved in our previous 
work \cite{GT}. This also provides a reconstruction procedure of a potential $V$ at any given point $z_0$ 
of $M_0$ where there exists a Morse holomorphic function with a critical point at $z_0$, see Remark \ref{rem1}.
Constructive methods have also been obtained by \cite{HeMi,HeNo} 
for isotropic conductivity on Riemann surfaces. 
The result in Theorem \ref{th2} is however not sufficient to identify the curvature connection and potential, 
and part of our work is to show that the Cauchy data space of $L$ determines the Cauchy data space of $P=D+V$ for 
a certain $V$ associated to $(X,q)$.\\

The last part of Theorem \ref{main} consist in showing that the integrals of $X_1-X_2$ along closed loops are in $2\piÊ\zz$ if 
$\mc{C}_{L_1}=\mc{C}_{L_2}$. This is done using parallel transport and unique continuation.\\  

The organisation of the paper is the following: in the first section, we construct some right inverses 
for the operators $\bar{\pl}$ and $\bar{\pl}^*$ on a manifold with boundary $M_0$. Then in the next section we
prove identification results for first order systems $D+V$ as explained above. The following section
is focused on boundary determination in the inverse problem for the operator $L$, and then we 
use this to reduce the problem on the magnetic Laplacian to the problem for a first order system $D+V$. The last section
deals with the holonomy identification and we end up with an appendix containing technical results.\\

\textbf{Acknowledgement}. C.G. is partially supported by grant ANR-09-JCJC-0099-01. L.T is partially supported by NSF Grant No. DMS-0807502 and thanks the Ecole Normale Sup\'erieure for its hospitality. We both thank Gunther Uhlmann for helpful discussions.
\end{section}

\begin{section}{The Cauchy-Riemann Operator on $M$}
Let $M$ be the interior of a compact Riemann surface with boundary $\bbar{M}$. 
%and let $S$ be a closed surface containing $\bbar{M}$ (for instance $S$ can be the double of $\bbar{M}$). 
The surface is equipped with a complex structure.  
The Hodge star operator $*$ acts on the cotangent bundle $T^*M$, its eigenvalues are 
$\pm i$ and the respective eigenspace $T_{1,0}^*M:=\ker (*+i{\rm Id})$ and $T_{0,1}^*M:=\ker(* -i{\rm Id})$
are sub-bundle of the complexified cotangent bundle $\cc T^*M$ and the splitting $\cc T^*M=T^*_{1,0}M\oplus T_{0,1}^*M$ holds as complex vector spaces. In local complex coordinate $z$ induced by the complex structure, $T_{1,0}M$ is spanned
by $dz$ while $T^{*}_{0,1}M$ is spanned by $d\bar{z}$.
If $\pi_{1,0},\pi_{0,1}$ are the respective projections 
from $T^*M$ onto $T^*_{1,0}M$ and $T^*_{0,1}M$,  the Cauchy-Riemann operators $\bar{\pl},\pl$ mapping functions to $1$-forms are defined by $\bar{\pl}:=\pi_{0,1}d$ and $\pl:=\pi_{1,0}d$ so that $d=\pl+\bar{\pl}$. The operators $\bar{\pl},\pl$ mapping $1$-forms
to $2$-forms are defined by $\bar{\pl}:=d\pi_{1,0}$ and $\pl:=d\pi_{0,1}$, and again $d=\pl+\bar{\pl}$.
Our main goal in this subsection is to construct some right inverses of the Cauchy-Riemann operators.
\begin{proposition}
\label{cauchy riemann}
There exists an operator $\bar{\pl}^{-1} : C_0^\infty(M, T^*_{1,0}M) \to C^{\infty}(M)$  which satisfies the following\\
(i)  $\bar{\pl}\bar{\pl}^{-1}\omega=\omega$ for all $\omega\in C_0^\infty(M,T^*_{1,0}M)$,\\
(ii) if $\chi_j\in C_0^\infty(M)$ are supported in some complex charts $U_i$ 
bi-holomorphic to a bounded open set $\Omega\subset \cc$ with complex coordinate $z$, and such that 
$\chi:=\sum_j\chi_j$ equal $1$ on $M_0$, then as operators
\[\bar{\pl}^{-1}\chi =\sum_{j}\hat{\chi}_j \bar{T} \chi_j +K\]
where $\hat{\chi}_j\in C_0^\infty(U_j)$ are such that $\hat{\chi}_j\chi_j=\chi_j$,  $K$ has a smooth 
kernel on $M\x M$ and $\bar{T}$ is given in the complex coordinate $z\in U_i\simeq \Omega$ by 
\[\bar{T}(fd\bar{z})=\frac{1}{\pi}\int_{\cc}\frac{f(z')}{z-z'}dz_1'dz_2'\]
where $dv_g(z)=\alpha^2(z)dz_1dz_2$ is the volume form of  $g$ in the chart.\\
(iii) $\bar{\pl}^{-1}$ is bounded from $L^p(T_{0,1}^*M)$ to $W^{1,p}(M)$ for any $p\in (1,\infty)$.
\end{proposition}
\begin{proof} The existence of a right inverse $\bar{\pl}^{-1}$ is proven in \cite[Th. C1.10]{McDSa} : 
by taking a totally real subbundle $F\subset \pl\bbar{M}\x\cc$ over the boundary $\pl \bbar{M}$ 
with a large boundary Maslov index (see \cite[Cor 2.2.]{GT2})
for an explicit $F$ having large Maslov index,
the operator $\bar{\pl}=W^{1,2}_F(M)\to L^2(T_{0,1}^*M)$ is Fredholm if $W^{1,2}_F(M)$ denotes the space of 
$L^2$ functions with one derivative in $L^2$ and boundary value in the bundle $F$, and moreover $\bar{\pl}$ is surjective if the Maslov index is chosen larger than $-2\chi(M)$ where $\chi(M)$ is the Euler characteristic of $M$. 
Moreover, $\bar{\pl}^{-1}$ maps $W^{k,2}(M)$ to $W^{k+1,2}(M,T^*_{0,1}M)$ for all $k\in \nn$ by elliptic regularity.
Observe that $\bar{\pl}^{-1}\bar{\pl}-1$ maps $W^{1,2}_F(M)$ into $\ker \bar{\pl}\cap W_{F}^{1,2}(M)$ which is a finite dimensional space
spanned by some smooth functions $\psi_1,\dots \psi_n$ (by elliptic regularity) on $M$. Assuming that $(\psi_j)_j$ is an orthonormal basis in $L^2$, this implies that, on $W^{1,2}_F(M)$
\[\bar{\pl}^{-1}\bar{\pl}=1-\Pi \, \textrm{ where }\, \Pi= \sum_{k=1}^n\psi_k\cjg\cdot,\psi_k\cjd_{L^2(M)}.\]
Now we also have
\[\bar{\pl}\sum_{j}\hat{\chi}_j\bar{T}\chi_j= \chi+\sum_{j}[\bar{\pl},\chi_j]\bar{T}\chi_j\]
and the last operator on the right has a smooth kernel in view of $\chi_j\nabla\hat{\chi}_j=0$ 
and the fact that $T$ has a smooth kernel outside the diagonal $z=z'$. 
Now since $\hat{\chi}_j\in C_0^\infty(M)\subset W^{1,2}_F(M)$, we can multiply by $\bar{\pl}^{-1}$ on the left of the last identity 
and obtain 
\[\bar{\pl}^{-1}\chi=\sum_{j}\hat{\chi}_j\bar{T}\chi_j-\Pi\sum_{j}\hat{\chi}_j\bar{T}\chi_j- \pl^{-1}\sum_{j}[\bar{\pl},\chi_j]\bar{T}\chi_j.\]  
The last two operator on the right have a smooth kernel on $M\x M$, in view 
of the smoothness of $\psi_k$ and the kernel of $\sum_{j}[\bar{\pl},\chi_j]\bar{T}\chi_j$, 
and since $\pl^{-1}$ maps $C_0^{\infty}(M)$ to $C^{\infty}(M,T_{0,1}^*M)$.  
\end{proof}

\begin{lemma}\label{inverseadjoint}
Let $\bar{\pl}^*=-i*\partial : W^{1,p}(T^*_{0,1}M)\to L^p(M)$, then there exists an operator  
$\bar{\pl}^{*-1}$ mapping $C_0^\infty(M)$ to $C^\infty(T_{0,1}^*M)$ 
which satisfies the following:\\ 
(i) $\bar{\pl}^*\bar{\pl}^{*-1}\omega=\omega$ for all $\omega\in C_0^\infty(M)$,\\
(ii) if $\chi_j\in C_0^\infty(M)$ are supported in some complex charts $U_i$ 
bi-holomorphic to a bounded open set $\Omega\subset \cc$ with complex coordinate $z$, and such that 
$\chi:=\sum_j\chi_j$ equal $1$ on $M_0$, then as operators
\[\bar{\pl}^{*-1}\chi =\sum_{j}\hat{\chi}_j T \chi_j +K\]
where $\hat{\chi}_j\in C_0^\infty(U_j)$ are such that $\hat{\chi}_j\chi_j=\chi_j$,  $K$ has a smooth 
kernel on $M\x M$ and $T$ is given in the complex coordinate $z\in U_i\simeq \Omega$ by 
\[Tf(z)=\Big(\frac{1}{\pi}\int_{\cc}\frac{f(z')}{\bar{z}-\bar{z}'}dv_g(z')\Big )d\bar{z}\]
where $dv_g(z)=\alpha^2(z)dz_1dz_2$ is the volume form of  $g$ in the chart.\\
(iii) $\bar{\pl}^{*-1}$ is bounded from $L^p(M)$ to $W^{1,p}(T_{0,1}^*M)$ for any $p\in (1,\infty)$.
\end{lemma}
\begin{proof} Let $G$ be the Green's kernel with Dirichlet condition on $\bbar{M}$. Then one has
$2\bar{\pl}^*\bar{\pl}G=1$ and $G$ maps $L^p(M)$ to $W^{2,p}(M)$ by elliptic regularity. Thus we shall set 
$\bar{\pl}^{*-1}:=2\bar{\pl} G$ which maps $L^p(M)$ to $W^{1,p}(T_{0,1}^*M)$. This proves (i) and (iii).
In local complex coordinate $z$ in each $U_i$, the metric has the local form $g=\alpha^2(z)|dz|^2$ for some positive function 
$\alpha(z)$, thus
$\Delta_g=\alpha^{-2}(z)\Delta_z$ where $\Delta_z=-4\pl_z\pl_{\bar{z}}$ is the Euclidean Laplacian. 
Therefore 
\[\Delta_g\sum_{j}\hat{\chi}_jG_0\alpha^2\chi_j=\chi+\sum_j [\Delta,\hat{\chi}_j]G_0\alpha^2\chi_j\]
if $G_0(z,z'):=-(2\pi)^{-1}\log|z-z'|$ is the Green's function on $\rr^2$. Since $\hat{\chi}_j\in C_0^\infty(M)$, we can multpliply 
this identity on the right by $(\bar{\pl}^*)^{-1}=2\bar{\pl} G$ and we deduce 
\[(\bar{\pl}^*)^{-1}\chi= 2\sum_{j}\hat{\chi}_j\pl_{\bar{z}}G_0\alpha^2\chi_j +2\sum_{j}[\bar{\pl},\hat{\chi}_j]G_0\alpha^2\chi_j
- 2\sum_j \bar{\pl} G[\Delta,\hat{\chi}_j]G_0\alpha^2\chi_j.\]
Since $\chi_j\nabla\hat{\chi}_j=0$ and $G_0$ is smooth outside the diagonal $z=z'$, we deduce that the last two terms have 
smooth kernel in view of the fact that $\bar{\pl} G$ preserves $C^\infty$. Now the operator $2\pl_{\bar{z}}G_0\alpha^2$ is equal to 
$T$ and we have proved (ii).
\end{proof}

Let $M_0$ be a surface with boundary included strictly in $M$ (for instance a deformation rectract of $M$) 
and for $q,p\in [1,\infty]$ 
let $\mc{E}$ be a linear extension operator from $M_0$ to $M$ which maps continuously 
$W^{k,p}(M_0,T^*_{0,1}M_0)$ to the set $W_{c}^{k,p}(M,T_{0,1}^*M)$ of compactly supported functions 
in $W_{k,p}(M,T_{0,1}^*M)$, 
for $k=0,1$, with a range made of functions with support inside the region $M_\delta:=\{m\in M; d(m,M_0)\leq \delta\}$ for some small $\delta>0$. Finally,  let $\mc{R}:L^q(M)\to L^q(M_0)$ be the restriction map from $M$ to $M_0$. 
   
\begin{lemma}\label{estimate1}
Let $\psi$ be a real valued smooth Morse function on $M$ and let $\bar{\pl}^{-1}_\psi:=\mc{R}\bar{\pl}^{-1}e^{-2i\psi/h}\mc{E}$ where $\bar{\pl}^{-1}$ 
is the right inverse of $\bar{\pl}:W^{1,p}(M)\to L^{p}(T^*_{0,1}M)$ constucted in Proposition \ref{cauchy riemann}. 
Let $q\in(1,\infty)$ and $p>2$,  then there exists $C>0$ independent of $h$ such that for all 
$\omega\in W^{1,p}(M_0,T^*_{0,1}M_0)$  
\begin{equation}\label{q<2} 
||\bar{\pl}^{-1}_\psi \omega||_{L^q(M_0)}\leq Ch^{2/3}||\omega||_{W^{1,p}(M_0,T^*_{0,1}M_0)} \, \, \,{\rm if }\, 1 \leq q<2 
\end{equation}
\begin{equation}\label{q>=2}
||\bar{\pl}^{-1}_\psi \omega||_{L^q(M_0)}\leq Ch^{1/q}||\omega||_{W^{1,p}(M_0,T^*_{0,1}M_0)} \,\,\, {\rm if }\, 2\leq q\leq p.
\end{equation}
There exists $\eps>0$ and $C>0$ such that for all $\omega\in W_c^{1,p}(M_0,T^*_{0,1}M_0)$
\begin{equation}\label{q=2}
||\bar{\pl}^{-1}_\psi \omega||_{L^2(M_0)}\leq Ch^{\demi+\eps}||\omega||_{W^{1,p}(M_0,T^*_{0,1}M_0)}. 
\end{equation}
\end{lemma}
\begin{proof}
Observe that the estimate \eqref{q=2} is a direct corollary of \eqref{q>=2} and \eqref{q<2} by using interpolation.
We recall the Sobolev embedding $W^{1,p}(M)\subset C^\alpha(M)$ for 
$\alpha\leq 1-2/p$ if $p>2$, and we shall denote by $T$ the Cauchy-Riemann inverse of $\pl_{\bar{z}}$ in $\cc$:
\[Tf(z):=\frac{1}{\pi}\int_{\cc}\frac{f(\xi)}{z-\xi}d\xi_1d\xi_2\]
where $\xi=\xi_1+i\xi_2$. If $\Omega, \Omega'\subset \cc$ are bounded open sets, then 
the operator $\indic_{\Omega'}T$ maps $L^p(\Omega)$ to $L^p(\Omega')$. 
Clearly, since $\mc{E},\mc{R}$ are continuous operators, it suffices to prove the estimates for compactly supported forms
$\omega\in W^{1,p}(T_{0,1}^*M)$ on $M$. Thus by partition of unity, it suffices to assume that $\omega$ is  compactly 
 supported in a chart biholomorphic to a bounded domain $\Omega\in \cc$, 
 and since the estimates will be localized, we can assume with no loss of generality that $\psi$ has only one critical point, say $z_0\in \Omega$ (in the chart).  
The expression of $\bar{\pl}^{-1}_\psi (fd\bar{z})$ in complex local coordinates in the chart $\Omega$ satisfies 
\[ \bar{\pl}^{-1}_\psi (f(z)d\bar{z})= \chi(z)T(e^{-2i\psi/h}f) + K(e^{-2i\psi/h}fd\bar{z})\]
where $K$ is an operator with smooth kernel and $\chi\in C_0^\infty(\cc)$.\\

Let us first prove \eqref{q<2}. Let $\varphi\in C_0^\infty(\cc)$ be a function which is equal to $1$ for $|z-z_0|>2\delta$ and 
to $0$ in $|z-z_0|\leq \delta$, where  $\delta>0$ is a parameter that will be chosen later (it will depend on $h$).
Using Minkowski inequality, one can write when $q<2$
\begin{equation}\label{1-varphi}
\begin{split}
||\chi T((1-\varphi)e^{-2i\psi/h}f)||_{L^q(\cc)} 
\leq & \int_{\Omega}\Big|\Big|\frac{\chi(\cdot)}{|\cdot -\xi|}\Big|\Big|_{L^q(\cc)}|(1-\varphi(\xi))f(\xi)| d\xi_1d\xi_2\\
\leq & C||f||_{L^\infty}\int_{\Omega}|(1-\varphi(\xi))| d\xi_1d\xi_2
\leq C\delta^2||f||_{L^\infty}.
\end{split}\end{equation}
On the support of $\varphi$, we observe that since $\varphi=0$ near $z_0$, we can use 
\[T(e^{-2i\psi/h}\varphi f)=\demi ih[e^{-2i\psi/h}\frac{\varphi f}{\bar{\pl}\psi}-T(e^{-2i\psi/h}\bar{\pl}(\frac{\varphi f}{\bar{\pl}\psi}))]\] 
and the boundedness of $T$ on $L^q$ to deduce that for any $q<2$
\begin{equation}\label{varphi}
\begin{split}
||\chi T(\varphi e^{-2i\psi/h}f)||_{L^q(\cc)} 
\leq & Ch \Big(||\frac{\varphi f}{\bar{\pl}\psi}||_{L^q} +
||\frac{f\bar{\pl}\varphi}{\bar{\pl}\psi}||_{L^q}+||\frac{\varphi \bar{\pl}f}{\bar{\pl}\psi}||_{L^q}+
||\frac{f\varphi}{(\bar{\pl}\psi)^2}||_{L^q}
\Big).
%\\
%\leq & Ch\delta^{\frac{2}{q}-2}||f||_{L^\infty}+ Ch\delta^{\frac{2}{q}-\frac{2}{p}-1}||f||_{W^{1,p}}\leq Ch\delta^{\frac{2}{q}-2}||f||_{W^{1,p}}.
\end{split}\end{equation}
The first term is clearly bounded by $\delta^{-1}\|f\|_{L^\infty}$ due to the fact that $\psi$ is Morse. For the last term, observe that since $\psi$ is Morse, $\frac{1}{|\partial \psi|} \leq \frac{c}{|z-z_0|}$ near $z_0$, therefore
\[||\frac{f\varphi}{(\bar{\pl}\psi)^2}||_{L^q} \leq C\|f\|_{L^\infty} (\int_\delta^1 r^{1-2q}dr)^{1/q} \leq C \delta^{\frac{2}{q} -2}\|f\|_{L^\infty}.\]
The second term can be bounded by $||\frac{f\bar{\pl}\varphi}{\bar{\pl}\psi}||_{L^q}\leq\|f\|_{L^\infty} ||\frac{\bar{\pl}\varphi}{\bar{\pl}\psi}||_{L^q}$. Observe that while $\|\frac{\bar{\pl}\varphi}{\bar{\pl}\psi}\|_{L^\infty}$ grows like $\delta^{-2}$, $\bar{\pl}\varphi$ is only supported in a neighbourhood of radius $2\delta$. Therefore we obtain
\[||\frac{f\bar{\pl}\varphi}{\bar{\pl}\psi}||_{L^q} \leq \delta^{2/q -2} \|f\|_{L^\infty}.\]
The third term can be estimated by
\[||\frac{\varphi \bar{\pl}f}{\bar{\pl}\psi}||_{L^q}\leq C||\bar{\pl}f||_{L^p} ||\frac{\varphi}{\bar{\pl}\psi}||_{L^\infty}\leq 
C\delta^{-1}||\bar{\pl}f||_{L^p}. \] 
Combining these four estimates with \eqref{varphi} we obtain
\[||\chi T(\varphi e^{-2i\psi/h}f)||_{L^q(\cc)}\leq h\|f\|_{W^{1,p}}(\delta^{-1} + \delta^{2/q -2}).\] 
Combining this and \eqref{1-varphi} and optimizing by taking $\delta=h^{1/3}$, we deduce that 
\begin{equation}
\label{singular part q<2}
||\chi T( e^{-2i\psi/h}f)||_{L^q(\cc)}\leq h^{2/3}\|f\|_{W^{1,p}}
\end{equation}
if $q<2$.
We now move on to the smoothing part given by $K(e^{-2i\psi/h}f)$. Take $\chi$ to be a compactly supported function in $\Omega$ such that it is equal to $1$ on the support of $f$, we see that $K(e^{2i\psi/h}f) = K(e^{-2i\psi/h}(f - \chi f(z_0)) + f(z_0)K(e^{-2i\psi/h}\chi)$. By applying stationary phase, we easily see that $\|f(z_0)K(e^{-2i\psi/h}\chi)\|_{L^q} \leq Ch \|f\|_{C^0}$ for any $q\in [1,\infty]$. 
For the first term, we write $\til{f}:=f-\chi f(z_0)$ and we integrate by parts to get, for some smoothing operator $K'$ 
\[K(e^{-2i\psi/h}\til{f}) = hK'(e^{-2i\psi/h}\til{f}) +\frac{h}{2i} K\Big(e^{-2i\psi/h}\pl_z\Big(\frac{\til{f}}{\pl_z\psi}\Big)\Big).\]
By the fact that $K$ and $K'$ are smoothing, we see that for all $k\in \nn$
\[\|K(e^{2i\psi/h}\til{f})\|_{C^k} \leq h C\Big(\|f\|_{L^\infty} + \Big\|\pl_z\Big(\frac{\til{f}}{\pl_z\psi}\Big)\Big\|_{L^1}\Big)\]
Using the fact that $\psi$ is Morse, the Sobolev embedding $W^{1,p}\subset C^{\alpha}$ for $\alpha=1-2/p$ 
and  $\til{f}(z_0)=0$, we can estimate the last term by $C \|f\|_{W^{1,p}}$ if $p>2$. Therefore,
\begin{eqnarray}
\label{smoothing estimate}
\|K(e^{2i\psi/h}f)\|_{L^q} \leq C h\|f\|_{W^{1,p}}
\end{eqnarray}
for any $q\in [1,\infty]$ and $p>2$. Combining \eqref{smoothing estimate} and \eqref{singular part q<2} we see that \eqref{q<2} is established.\\

Let us now turn our attention to the case when $\infty>q\geq 2$, one can use the boundedness of $T$ on $L^q$ and thus 
\begin{equation}\label{casq>2}
||\chi T((1-\varphi)e^{-2i\psi/h}f)||_{L^q(\cc)}\leq  ||(1-\varphi)e^{-2i\psi/h}f||_{L^q(\Omega)}\leq C\delta^{\frac{2}{q}}||f||_{L^\infty}.
\end{equation}
Now since $\varphi=0$ near $z_0$, we can use 
\[T(e^{-2i\psi/h}\varphi f)=
\demi ih[e^{-2i\psi/h}\frac{\varphi f}{\pl_{\bar{z}}\psi}-T(e^{-2i\psi/h}\pl_{\bar{z}}(\frac{\varphi f}{\pl_{\bar{z}}\psi}))]\] 
and the boundedness of $T$ on $L^q$ to deduce that for any $q\leq p$, \eqref{varphi} holds again with all the terms satisfying the same estimates as before so that 
\[\|T(e^{-2i\psi/h}\varphi f)\|_{L^q} \leq Ch\|f\|_{W^{1,p}}(\delta^{2/q -2} + \delta^{-1}) \leq Ch\delta^{2/q-2}\|f\|_{W^{1,p}} \]
since now $q \geq 2$. Now combine the above estimate with \eqref{casq>2} and take $\delta = h^\demi$ we get
\[\|T(e^{-2i\psi/h} f)\|_{L^q} \leq h^{1/q} \|f\|_{W^{1,p}}\] 
for $2 \leq q\leq p$. The smoothing operator $K$ is controlled by \eqref{smoothing estimate} for all $q\in [1,\infty]$ and therefore we obtain \eqref{q>=2}.
\end{proof}

Observe that the adjoint $\mc{R}^*:L^2(M_0)\to L^2(M)$ of $\mc{R}:L^2(M)\to L^2(M_0)$ is simply given by  
$\mc{R}^*f=\indic_{M_0}f$ where $\indic_{M_0}$ is the characteristic function of $M_0$ in $M$. 
In particular $\mc{R}^*V\in W^{1,p}(M)$ for $p>2$ if $V|_{\pl M_0}=0$ and $V\in W^{1,p}(M_0)$. By Proposition \ref{cauchy riemann}, 
the operator $(\bar{\pl}^{-1})^*$ satisfies 
\[\chi(\bar{\pl}^{-1})^*=\sum_j \chi_j\bar{T}^*\hat{\chi}_j+K^*\] 
where $K^*$ has a smooth integral  kernel on $M\x M$ and $\bar{T}^*=T$. The proof of Lemma \ref{estimate1} can then be 
applied in the same way to  deduce that for $v\in W^{1,p}(M_0)$ with $v|_{\pl M_0}=0$ 
\begin{equation}\label{estadj}
||\mc{E}^* (\bar{\pl}^{-1})^*R^*(e^{-2i\psi/h}v)||_{L^2(T^*_{0,1}M_0)}\leq Ch^{\demi+\eps}||v||_{W^{1,p}(M_0)}
\end{equation}
where we used that $\chi\mc{E}=\mc{E}$ if $\chi=1$ on $M_\delta$.
The same following estimate also holds if $w\in W^{1,p}(M_0, T_{0,1}^*M_0)$ with $w|_{\pl M_0}=0$
\begin{equation}\label{estadj2}
||\mc{E}^* (\bar{\pl}^{*-1})^*R^*(e^{2i\psi/h}w)||_{L^2(M_0)}\leq Ch^{\demi+\eps}||w||_{W^{1,p}(M_0,T_{0,1}^*M_0)}.
\end{equation}

Similarly, let $\mc{R}$ also denotes the restriction of section of $T^*_{0,1}M$ to 
$M_0$ and if $\mc{E}$ is an extension map from $M_0$ to $M$ which is continuous
from $W^{k,p}(M_0)$ to $W^{k,p}_c(M)$ for $k=0,1$ and with range some functions 
having support in $M_\delta$. One has 
\begin{lemma}
Let $\psi$ be a smooth real valued Morse function on $M$ and let 
$\bar{\pl}_\psi^{*-1}:=\mc{R}\bar{\pl}^{*-1}e^{2i\psi/h}\mc{E}$ where $\bar{\pl}^{*-1}$ 
is the right inverse  constucted in Proposition \ref{inverseadjoint} for $\bar{\pl}^*:W^{1,p}(M,T^*_{0,1}M)\to L^p(M)$. 
Let $q\in(1,\infty)$ and $p>2$,  then there exists $C>0$ independent of $h$ such that for all 
$\omega\in W^{1,p}(M_0)$  
\begin{equation}\label{q<2 2} 
||\bar{\pl}^{*-1}_\psi \omega||_{L^q(M_0,T^*_{0,1}M_0)}\leq Ch^{2/3}||\omega||_{W^{1,p}(M_0)} \, \, \,{\rm if }\, 1 \leq q<2 
\end{equation}
\begin{equation}\label{q>=2 2}
||\bar{\pl}^{*-1}_\psi \omega||_{L^q(M_0,T^*_{0,1}M_0)}\leq Ch^{1/q}||\omega||_{W^{1,p}(M_0)} \,\,\, {\rm if }\, 2\leq q\leq p.
\end{equation}
There exists $\eps>0$ and $C>0$ such that for all $\omega\in W_c^{1,p}(M_0)$
\begin{equation}\label{q=2 2}
||\bar{\pl}^{*-1}_\psi \omega||_{L^2(M_0,T^*_{0,1}M_0)}\leq Ch^{\demi+\eps}||\omega||_{W^{1,p}(M_0)}. 
\end{equation}
\end{lemma}
\begin{proof} The proof is exactly the same as the proof of Lemma \ref{estimate1}. We do not repeat it.
\end{proof}

\end{section}

\begin{section}{Solutions to First Order Systems}\label{firstorder}
Let $M_0$ be a surface with boundary included strictly in $M$ (for instance a deformation retract of $M$)  
and let $\Phi = \phi + i\psi$ be a Morse holomorphic function on $M$.
Such $\Phi$ exist by Corollary 2.2. in \cite{GT}. 
We shall denote respectively by 
\[V = \begin{pmatrix}{v}&{0}\cr{0}&{v'}\end{pmatrix} \,\,\, \textrm{ and }D=\begin{pmatrix}{0}&{\bar{\pl}^*}\cr{\bar{\pl}}&{0}\end{pmatrix}\]
the matrix potential where $v,v'\in W^{1,p}(M_0)$ (with $p>2$) are complex valued and the Dirac type operator, acting on sections of  
the bundle $\Sigma:=\Lambda^0(M_0)\oplus\Lambda^{0,1}(M_0)$ over $M_0$. 
In this section, we will construct geometric optic solutions $F \in W^{1,p}(\Sigma)$ (also called Faddeev type solutions) 
which solve the equation 
\[(D+V)F=0\]  
on $M_0$.
It is clear that 
\[D=\left(\begin{array}{cc}
 e^{-\bar{\Phi}/h} & 0\\
0 & e^{-\Phi/h} 
\end{array}\right)D\left(\begin{array}{cc}
 e^{\Phi/h} & 0\\
0 & e^{\bar{\Phi}/h} 
\end{array}\right)\] 
and thus 
\[\begin{gathered}\left(\begin{array}{cc}
 e^{-\bar{\Phi}/h} & 0\\
0 & e^{-\Phi/h} 
\end{array}\right)(D+V)\left(\begin{array}{cc}
 e^{\Phi/h} & 0\\
0 & e^{\bar{\Phi}/h} 
\end{array}\right)= D+V_\psi\,,\\ 
\, V_\psi:=\left(\begin{array}{cc}
 e^{2i\psi/h}v & 0\\
0 & e^{-2i\psi/h}v' 
\end{array}\right).\end{gathered}\]
We will then construct solutions $F_h$ of $(D+V_\psi)F_h=0$ which have the form  
\[F_h=\left(\begin{array}{c}
a + r_h\\
b +s_h 
\end{array}\right)=:A+Z_h\]
where $a$ is some holomorphic functions on $M$, $b$ some anti-holomorphic $1$-form 
and $(r_h,s_h)\in W^{1,p}(\Sigma)$ which decays appropriately as $h\to 0$. 
In particular, we need to solve the system 
\[(D+V_\psi)Z_h=-V_\psi A=-\left(\begin{array}{c}
 e^{2i\psi/h}va \\
e^{-2i\psi/h} v'b
\end{array}\right).\]
Let us define the operators $D^{-1}$ and  $D_\psi^{-1}$ acting on  $\Lambda^{0}(M_0)\oplus \Lambda^{0,1}(M_0)$ by 
\[D^{-1}:=\left(\begin{array}{cc}
0 & \mc{R}\bar{\pl}^{-1}\mc{E}\\
\mc{R}\bar{\pl}^{*-1}\mc{E}& 0 
\end{array}\right),\quad 
\, \, \, D_\psi^{-1}:= \left(\begin{array}{cc}
0 & \bar{\pl}^{-1}_\psi\\
\bar{\pl}_\psi^{*-1}& 0 
\end{array}\right)\]
which satisfy $DD^{-1}={\rm Id}$ on $L^q(M_0)$ for all $q\in (1,\infty)$ 
and $DD^{-1}_\psi V=V_\psi$.
To construct $Z_h$ solving $(D+V_\psi)Z_h=-V_\psi A$ in $M_0$, it then suffices to solve
\[(I+ D_\psi^{-1}V)Z_h=-D_\psi^{-1}VA.\]
Writing the components of this system explicitly we get
\begin{equation}
\label{s-r equation}
\left\{\begin{array}{ll}
r_h + \bar{\pl}_\psi^{-1}(v's_h) = -\bar{\pl}^{-1}_\psi (v'b)\\
s_h +\bar{\pl}_\psi^{*-1}(vr_h) = -\bar{\pl}_\psi^{*-1}(va)
\end{array}\right..
\end{equation}
Observe that since we are allowed to choose any holomorphic function $a$ and anti-holomorphic 1-form $b$, we may set $a=0$ in \eqref{s-r equation} and solve for $r_h$ to get
\begin{eqnarray}
\label{solve for r}
(I - S_h) r_h = -\bar{\pl}_\psi^{-1}(v' b) \,\,\, \textrm{ with }S_h:= \bar{\pl}_\psi^{-1}v'\bar{\pl}_\psi^{*-1}v .
\end{eqnarray}
where $v,v'$ are viewed as multiplication operators. 
We have the following lemma:
\begin{lemma}\label{normestim}
Let $p>2$ and assume that $v\in L^\infty(M_0)$ and $v'\in W^{1,p}(M_0)$, 
then $S_h$ is bounded on $L^r(M_0)$ for any $1<r\leq p$ and satisfies $||S_h||_{L^r\to L^r}=O(h^{1/r})$ if 
$r>2$ and $||S_h||_{L^2\to L^2}=O(h^{1/2-\eps})$ for any $0<\eps<1/2$ small.
\end{lemma}
\begin{proof}
First, notice that $\bar{\pl}_\psi^{*-1}$ maps $L^r(M_0)$ to $W^{1,r}(M_0,T^*_{0,1}M_0)$ with norm $O(1)$ as $h\to 0$
by (iii) in Lemma \ref{inverseadjoint} and the properties of $\mc{E},\mc{R}$. Therefore, if $v'\in W^{1,p}(M_0)$ and $v\in L^\infty(M_0)$, the operator 
$v'\bar{\pl}_\psi^{*-1}v$ maps $L^r$ to $W^{1,r}(M_0,T^*_{0,1}M_0)$ with norm $O(1)$ for $r\leq p$ and 
Lemma \ref{estimate1} can be used  to deduce that  $S_h$ maps $L^r$ to $L^r$ with norm $O(h^{1/r})$ if $r>2$.
If $r<2$, $v'\bar{\pl}_\psi^{*-1}$ maps $L^r(M_0)$ to $L^r(T^*_{0,1}M_0)$ with norm $O(1)$, and $\bar{\pl}^{-1}_\psi$ maps
$L^r(T^*_{0,1}M_0)$ to $L^r(M_0)$ with norm $O(1)$, and thus $S_h$ is bounded on $L^r(M_0)$ with norm $O(1)$.
For all $\eps>0$ small, interpolating between $r=1+\eps$ and $r=2+\eps$, gives the desired result for $r=2$. 
\end{proof}
In view of Lemma \ref{normestim}, equation \eqref{solve for r} can be solved by using Neumann series by setting (for small $h>0$)
\begin{equation}
\label{r neumann}
r_h := -\sum\limits_{j = 0}^\infty S_h^j\bar{\pl}_\psi^{-1}v' b
\end{equation}
as an element of any $L^q(M_0)$ for $q\geq 2$.
Substituting this expression for $r$ into equation \eqref{s-r equation} when $a=0$, we get that
\begin{equation}\label{shnorm}
s_h = -\bar{\pl}_\psi^{*-1}vr_h.
\end{equation}
We now derive the asymptotics in $h$ for $s_h$ and $r_h$.
\begin{lemma}
\label{s-r decay}
If $v\in L^\infty(M_0)$ and $v'\in W^{1,p}(M_0)$ for some $p>2$, then there exists $\epsilon >0$ such that
\[\|s_h\|_{L^2(M_0)} + \|r_h\|_{L^{2}(M_0)}=O(h^{\demi+\epsilon})\]
%and 
%\[s = -I_{M_0}P_{\psi}' E_{M_0\to M}( qI_{M_0}\bar P_{-\psi} E_{M_0\to M}q' b) + o_{L^{q}}(h)\]
%for some $q<2$.
\end{lemma}
\begin{proof}
The statement for $r_h$ is an easy consequence of Lemma \ref{estimate1} and \ref{normestim}: 
indeed $||\bar{\pl}^{-1}_\psi v'b||_{L^2}=O(h^{\demi+\eps})$ 
by Lemma \ref{estimate1} and $||S_h||_{L^2\to L^2}=O(h^{\demi-\eps})$ thus $||r_h||_{L^2}=O(h^{\demi+\eps})$.
The estimate for $s_h$ comes from the fact that $||\bar{\pl}^{*-1}_\psi||_{L^2\to L^2}=O(1)$ and \eqref{shnorm}. 
\end{proof}
The same method can clearly be used by setting $b=0$ and solving for $s_h$ first.
We summarize the results of this section into the following proposition
\begin{proposition}
\label{first order CGO a = 0}
Let $\Phi = \phi + i\psi$ be a Morse holomorphic function on $M$, and 
$b$ an anti-holomorphic 1-form on $M$. If $v\in L^\infty(M_0)$ and $v'\in W^{1,p}(M_0)$ for some $p>2$, then
there exist solutions to $(D+V)F= 0$ on $M_0$ of the form
\begin{equation}\label{Fh} 
F_h = \begin{pmatrix}{e^{\Phi/h} r_h}\cr{e^{\bar\Phi/h}(b + s_h)}\end{pmatrix}
\end{equation}
where $\|s_h\|_{L^2(M_0)} + \|r_h\|_{L^{2}(M_0)}=O(h^{\demi+\eps})$ for some $\eps>0$. 
If conversely  $v'\in L^\infty(M_0)$ and $v\in W^{1,p}(M_0)$ for some $p>2$, then
there exist solutions to $(D+V)G= 0$ on $M_0$ of the form
\begin{equation} \label{Gh}
G_h = \begin{pmatrix}{e^{\Phi/h} (a+r_h)}\cr{e^{\bar\Phi/h}s_h}\end{pmatrix}
\end{equation}
where $\|s_h\|_{L^2(M_0)} + \|r_h\|_{L^{2}(M_0)}=O(h^{\demi+\eps})$ for some $\eps>0$. 
%and
%\[s = -I_{M_0}P_{\psi}' E_{M_0\to M}( qI_{M_0}\bar P_{-\psi} E_{M_0\to M}q' b) + o_{L^{q}}(h).\]
\end{proposition}
As a corollary we obtain the
\begin{theorem}\label{inverseresult}
Let $p>2$ and $V_j\in L^\infty(M_0,{\rm End}(\Sigma))$ be some complex  
diagonal endomorphisms of $\Sigma$ for $j=1,2$ with diagonal entries $v_j\in L^\infty(M_0) ,v'_j\in L^\infty(M_0)$. 
Denote by 
\[\mc{C}_{V_j}:=\{i_{\pl M_0}^*F\in H^{\demi}(\pl M_0;\Sigma);  F\in H^1(M_0,\Sigma), (D+V_j)F=0\}\] 
the Cauchy data space of $D+V_j$, where $i_{\pl M_0}: \pl M_0\to M_0$ denotes the natural inclusion map.\\
(i) If $v'_j\in W^{1,p}(M_0)$ and  $\mc{C}_{V_1}=\mc{C}_{V_2}$   
then $v'_1=v'_2$.\\
(ii) if  $v_j\in W^{1,p}(M_0)$ and  $\mc{C}_{V_1}=\mc{C}_{V_2}$   
then $v_1=v_2$.\\
As a consequence, if $V_j\in W^{1,p}(M_0,{\rm End}(\Sigma))$, 
and $\mc{C}_{V_1}=\mc{C}_{V_2}$, then $V_1=V_2$. 
\end{theorem}
\begin{proof} Let $\Phi$ be a Morse holormophic function with a critical point at $z_0$. The existence of such a function for a dense set of points $z_0$ of $M_0$ is insured by Proposition 2.1 of \cite{GT}.
We start by defining the respective solutions 
\[
F^1_h:=\begin{pmatrix}{e^{\Phi/h} r^1_h}\cr{e^{\bar\Phi/h}(b + s^1_h)}\end{pmatrix},  \quad 
F^2_h:=\begin{pmatrix}{e^{-\Phi/h} r^2_h}\cr{e^{-\bar\Phi/h}(b + s^2_h)}\end{pmatrix}
\]
of $(D+V_1)F^1_h=0$  and $(D+V_2^*)F^2_h=0$ where $r^j_h,s^j_h$ are constructed 
in Proposition \ref{first order CGO a = 0}. 
Since $\mc{C}_1=\mc{C}_2$, there exists $F_h$ solution of $(D+V_2)F_h=0$ such that 
$i_{\pl M_0}^*F_h=i_{\pl M_0}^*F^1_h$. In particular, $(D+V_2)(F^1_h-F_h)=(V_2-V_1)F^1_h$ and 
$i_{\pl M_0}^*(F^1_h-F_h)=0$. Then using Green's formula and the vanishing of $F^1_h-F_h$ on the boundary
\begin{equation}\label{integid}
0=\int_{M_0}\cjg (D+V_2)(F^1_h-F_h),F_h^2\cjd =\int_{M_0}\cjg (V_2-V_1)F^1_h,F^2_h\cjd.
\end{equation}
where $\cjg\cdot,\cdot\cjd$ denotes the natural Hermitian scalar product on $\Sigma$ induced by $g$.
This gives 
\begin{equation}\label{integralid}
0=\int_{M_0}(v_2'-v_1')e^{-2i\psi/h}\Big(|b|^2+ \cjg b,s_h^2\cjd+\cjg s_h^1,b\cjd \Big)+ 
(v_2-v_1)e^{2i\psi/h}r^1_h\bbar{r^2_h}.
\end{equation}
First, notice that by Proposition \ref{first order CGO a = 0}, 
\begin{equation}\label{firstremain}
\int_{M_0}(v_2-v_1)e^{2i\psi/h}r^1_h\bbar{r^2_h}=O(h^{1+\eps})
\end{equation}
for some $\eps>0$.
Then we choose 
$b$ to be an anti-holomorphic $1$-form which vanishes 
at all critical points of $\Phi$ in $M_0$ except at the critical point $z_0\in M_0$ of $\Phi$. 
This can be done by using Riemann-Roch theorem (see Lemma 4.1 in \cite{GT}). 
We observe by using stationary phase that   
\begin{equation}\label{statphase}
\int_{M_0}(v_1'-v_2')e^{2i\psi/h}|b|^2 = C_{z_0} h e^{2i\psi(z_0)/h}(v'_1(z_0)-v'_2(z_0))|b(z_0)|^2+o(h)
\end{equation}
for some constant $C_{z_0}\not=0$. More precisely, to show this, it suffices to decompose 
$v':=v'_1-v'_2$ as $(v'-\chi v'(z_0))+\chi v'(z_0)$ where $\chi\in C_0^\infty(M_0)$
is supported near $p$, then we apply stationary phase to the term $\chi v'(z_0)$ and use 
integration by parts for the other term: if $\sum_{j}\chi_j=1$ is a partition of unity on $M_0$ associated to charts 
$U_j$ with complex coordinate $z$
\[\int_{M_0} e^{2i\psi/h}\chi_j(v'-\chi v'(z_0))|b|^2=\frac{h}{2i}\int_{U_j}
e^{2i\psi/h}\pl_z\Big(\frac{\chi_j(v'-\chi v'(z_0))|b|^2}{\pl_z \psi} \Big)\] 
since  $v'|_{\pl M_0}=0$ by the boundary identification of Lemma \ref{boundarydet}, 
and we finally conclude using Riemann-Lebesgue to deduce that the
right hand side term is $o(h)$ since $\pl_z(\chi_j(v'-\chi v'(z_0))|b|^2/\pl_z \psi)\in L^1(M_0)$ if $v'\in W^{1,p}$ for $p>2$. 
Let us now consider the term with $\cjg b,s_h^2\cjd$ in \eqref{integralid}: using \eqref{shnorm}
\[\int_{M_0}\cjg e^{2i\psi/h}  v'b,s_h^2\cjd=\int_{M_0} \cjg  \mc{E}^*(\bar{\pl}^{*-1})^*\mc{R}^*e^{2i\psi/h}v'b ,  v_2r^2_h \cjd.\]
Since $v'|_{\pl M_0}=0$, we may use \eqref{estadj2} to deduce that  
$||\mc{E}^*(\bar{\pl}^{*-1})^*\mc{R}^*e^{2i\psi/h}v'b||_{L^2}=O(h^{1/2+\eps})$ and thus combining with Proposition 
\ref{first order CGO a = 0}, we deduce that  
\[\int_{M_0}\cjg e^{2i\psi/h}  v'b,s_h^2\cjd=O(h^{1+\eps}).\]
the same argument gives that the term involving $\cjg s_h^1,b\cjd$ in \eqref{integralid} is $O(h^{1+\eps})$. 
These last two estimates combined with \eqref{statphase} and \eqref{firstremain} imply that $v'_1(z_0)=v'_2(z_0)$ by letting $h\to 0$.
The same proof using the complex geometric optics $G_h$ of Proposition \ref{first order CGO a = 0} gives 
$v_1=v_2$ if $v_j\in W^{1,p}(M_0)$ and $v_j\in L^\infty(M_0)$.
\end{proof}

\begin{remark}\label{rem1} 
As noted in section 4 of \cite{Bu}, this methods allows to get an inversion (or reconstruction) procedure to 
recover the value of a matrix potential $V$ at a given $z_0\in M_0$, provided we know a Morse holomorphic function $\Phi$
with a critical point at $z_0$ and $\Phi(z_0)=0$. 
We do not give details since it is essentially the same idea as \cite{Bu}, but essentially the method
is to compare to case $V_1=V$ to the free case $V_2=0$ 
and use  complex geometric optic (or Faddeev type) solutions $F_h^1,F_h^2$ for $h\to 0$, together with 
the  Green formula as we did above for identification: the boundary term is not zero anymore but 
is the information we measure and therefore multiplying by $h^{-1}$ and letting $h\to 0$, the boundary term
converges to $V(z_0)$ times an explicit non-zero constant .  
\end{remark}

\end{section}

\begin{section}{Boundary Determination}
For smooth $(X,q)$, it was shown by Nakamura-Sun-Uhlmann in \cite{NaSuUh} that the Cauchy data space determines the boundary values of $X_j$ up to an exact form. This is relaxed to regularity $X\in C^1,q\in L^\infty$ by Brown-Salo \cite{BrSa}. We summarize it in
\begin{proposition}
\label{boundary determination}
Let $X_1,X_2\in W^{2,p}(M_0,T^*M_0)$  and $q_1,q_2\in W^{1,p}(M_0)$ for some $p>2$, then if $\mc{C}_{L_1}=\mc{C}_{L_2}$
then 
%there exists a real valued function $\zeta\in W^{2,p}(\pl M_0)$ such that 
$i_{\pl M_0}^*X_1 = i_{\pl M_0}^*X_2$  and $q_1|_{\pl M_0} = q_2|_{\pl M_0}$, where $i_{\pl M_0}$ is the inclusion map of $\pl M_0$ into $M_0$. 
\end{proposition}
This statement was only shown in \cite{NaSuUh,BrSa} for $M = \Omega\subset\R^n$ but since the proof is only localized near a neighbourhood of the boundary, it adapts naturally on a general Riemann surface and we will not provide a proof here.
Notice that by adding an exact form  $d\zeta$ to $X_1$ with $\zeta$ a function vanishing on the boundary, we do not change the Cauchy data space $\mc{C}_{L_1}$. If $x$ is a boundary defining function such that $|dx|_g=1$ at $\pl M_0$, we can set $\zeta=xf(x)$ for some $C^1$ function $f$ and we have $d\zeta|_{\pl M_0}=f|_{\pl M_0}dx$; therefore if $\nu$ is the unit interior pointing normal vector field to $\pl M_0$, we have $d\zeta(\nu)|_{\pl M_0}=f|_{\pl M_0}$ and choosing $f$ accordingly, we can choose $\zeta$ so that $X_1+d\zeta=X_2$ at $\pl M_0$. By the gauge invariance, we can at best identify $X$ up to exact forms, and therefore we may assume that $X_1=X_2$ at $\pl M_0$ as forms on $M_0$, possibly by modifying $X_1$ through an exact form $d\zeta$.\\

For our purpose we will need additional information along the boundary in order to reduce our problem to a first order system. That is, we will show that the boundary value of certain primitives of the forms $X_j$ agree with that of the boundary value of a holomorphic function. More precisely, let $X_j\in \Lambda^1(M_0)$ 
and $A_j := \pi_{0,1}X_j$ and $B_j := \pi_{1,0}X_j$ where $\pi_{0,1}:\Lambda^{1}(M_0)\to \Lambda^{0,1}(M_0)$, and 
$\pi_{1,0}:\Lambda^{1}(M_0)\to \Lambda^{1,0}(M_0)$ are the natural projections.
The main result of this section is
\begin{proposition}
\label{boundary of holomorphic}
Let $\alpha_1$ and $\alpha_2$ be smooth functions such that $\bar\partial \alpha_j = A_j$. 
Then $e^{-i(\alpha_1-\alpha_2)}|_{\partial M}$ is the boundary value of a holomorphic function.
\end{proposition}

In order to prove the Proposition, we shall need a few Lemma characterizing boundary values of holomorphic functions.
Let us denote by $i_{\pl M_0}:\pl M_0\to M_0 $ the inclusion map.

\begin{lemma}
\label{orthogonality bc}
Let $f\in H^{1/2}(\partial M_0)$ be a complex valued function. Then $f$ is the restriction of a holomorphic function if and only if
\[\int_{\partial M} f i_{\partial M_0}^*\eta = 0\]
for all 1-forms $\eta\in C^{\infty}(M_0 ; T^*_{1,0}M_0)$ satisfying $\bar\partial\eta = 0$.
Similarly, $f\in H^{1/2}(\partial M_0)$ is the restriction of an anti-holomorphic function if and only if
\[\int_{\partial M_0} f i_{\partial M_0}^*\eta = 0\]
for all 1- forms $\eta\in C^{\infty}(M_0 ; T^*_{0,1}M_0)$ satisfying $\partial\eta = 0$.
\end{lemma}
\begin{proof}
We will only prove the holomorphic statement, the anti-holomorphic statement follows similarly. 
Suppose $f\in H^{1/2}(\partial M_0)$ is such that
\[\int_{\partial M} f i_{\partial M_0}^*\eta = 0.\]
and denote by $u\in H^1(M_0)$ its harmonic extension to $M_0$. 
We would like to show that $u$ is actually holomorphic. We will do this by showing that 
\[\langle\bar\partial u,\omega\rangle = 0,\,\, \forall \omega \in C^{\infty}(M_0, T^*_{0,1}M_0).\]
By the Hodge-Morrey decomposition given in \cite[Th 2.4.2]{Sch}, a 1-form $\omega\in L^2$ can be decomposed as 
\[\omega=d\alpha+* d\beta+\omega_0\]
where $\alpha,\beta \in H^1(M_0)$ satisfy Dirichlet condition $\alpha|_{\pl M_0}=\beta|_{\pl M_0}=0$ and 
$\omega_0$ is closed and co-closed $d\omega_0=0, d*\omega_0=0$.
If $\omega \in C^{\infty}(M_0,T^*_{0,1}M_0)$, then 
\[\omega = \pi_{0,1}\omega = \pi_{0,1} d\alpha + \pi_{0,1}* d\beta + \pi_{0,1}\omega_0\]
Since $\pi_{0,1} d = \bar\partial$ and $\pi_{0,1}* d = -i\bar\partial$ on functions, 
and $\eta:=\pi_{0,1}\omega_0 \in C^{\infty}(M_0,T^*_{0,1}M_0)$ 
satisfies $\partial \eta= 0$, we can write $\omega = \bar\partial \gamma +\eta$
where $\gamma$ has Dirichlet boundary condition and $\partial\eta= 0$. 
We now compute by Stoke's Theorem
\[\langle \bar\partial u,\omega\rangle = \cjg\bar{\pl}u,\bar\pl \gamma\cjd + \cjg  \bar\partial u,\eta\cjd = \langle  
\bar{\pl}^*\bar\partial u,\gamma \rangle -i \int_{\partial M} f i_{\partial M_0}^* \bar\eta .\]
The first term on the right side vanishes since $u$ is harmonic,  the last term vanishes 
since $\bar\eta \in C^{\infty}(M_0, T^*_{1,0}M_0)$ satisfies $\bar\pl \bar\eta = 0$ 
and we have assumed that $f$ is orthogonal to all such boundary values of sections of $T^*_{1,0}M_0$. 
Therefore we conclude that $\bar\partial u= 0$. 
Since we will not actually use the converse statement in this paper, it is left as an exercise.
\end{proof}
\begin{lemma}
\label{no boundary crit point}
Let $f\in H^{1/2}(\partial M_0)$ be a complex valued function. Then $f$ is the restriction of a holomorphic function if and only if
\[\int_{\partial M_0} f i_{\partial M_0}^*\partial \phi = 0\]
for all smooth real valued harmonic functions $\phi$ which have no critical point on the boundary.
Similary, $f\in H^{1/2}(\partial M_0)$ is the restriction of an anti-holomorphic function if and only if
\[\int_{\partial M_0} f i_{\partial M_0}^*\bar{\pl}\phi = 0\]
for all smooth real valued harmonic functions $\phi$ which have no critical point on the boundary.
\end{lemma}
\begin{proof}
By applying \cite[Th. C1.10]{McDSa} like in \cite[Cor. 2.3]{GT2} with a totally real 
subbundle boundary condition having high boundary Maslov index, one obtains that the 
operator $\partial : H^4(M_0)\to H^{3}(M_0; T^*_{1,0}M_0)$ is surjective. Using this we see that Lemma \ref{orthogonality bc} implies that $f$ is the restriction of a holomorphic function if and only if
\[\int_{\partial M_0} f i_{\partial M_0}^*\partial \phi = 0\]
for all harmonic functions $\phi$. Now it remains to show that this statement is equivalent to the case where we consider only smooth harmonic functions with no critical points on the boundary. That is, we want to show that smooth harmonic functions with no critical points on the boundary form a dense subset of the harmonic functions in $C^k(M_0)$. Indeed, let $\phi$ be a harmonic function with smooth boundary value $g\in C^\infty(\partial M_0)$. Since Morse functions are generic on the circle, it suffices to consider the case where $g$ is a Morse function with isolated critical points $\{x_1,..,x_N\}$. Clearly, the critical points of $\phi$ forms a subset of the set $\{x_1,..,x_N\}$. We will make a small perturbation to $\phi$ so that $x_1$ is guaranteed to not be a critical point.
Let $\phi_1$ be a smooth harmonic function with boundary value $g_1 \in C^\infty(\partial M_0)$ and such that 
and $d\phi_1(x_j)=d\nu$ for all $j=1,\dots,N$ if $d\nu$ is the unit conormal form to $\pl M_0$; 
the existence of such a function $\phi_1$ is insured by applying Lemma 2.6 of \cite{GT2} on the manifold $M$ containing strictly $M_0$:  indeed this Lemma says that there exist holomorphic functions with prescribed Taylor expansion to order $2$ at $x_1,\dots,x_N\in {\rm int}(M)$  and therefore taking its real part one obtains the desired harmonic function. 
For all $\eps >0$ small, the points $x_1,\dots,x_N$ are not critical point of the function $\phi' := \phi + \eps \phi_1$.  
Furthermore, for all $\eps>0$ small enough, $g' := \phi'|_{\pl M_0}$ is again a Morse function on $\pl M_0$ with critical points 
$\{x_1,..,x_N\}$, therefore the critical points of $\phi'$ on $\pl M_0$ are contained in $\{x_1,..,x_N\}$, which implies 
that $\phi'$ has no critical points on $\pl M_0$. Thus smooth harmonic functions with no critical points on $\pl M_0$ 
form a dense subset of $C^k$ harmonic functions and we are done.
\end{proof}
In view of Lemma \ref{no boundary crit point} and the fact that Morse harmonic functions form a dense subset of harmonic functions (see \cite[Lemma 2.2]{GT}) in $C^k(M_0)$, we have the following corollary:
\begin{coro}
\label{reduction}
Let $f\in H^{1/2}(\partial M_0)$ be a complex valued function. Then $f$ is the restriction of a holomorphic function if and only if
\[\int_{\partial M_0} f i_{\partial M_0}^*\partial \phi = 0\]
for all smooth real valued harmonic functions $\phi$ which is Morse up to the boundary and has no critical point on the boundary.
Similary, $f\in H^{1/2}(\partial M_0)$ is the restriction of an anti-holomorphic function if and only if
\[\int_{\partial M_0} f i_{\partial M_0}^*\bar{\pl}\phi = 0\]
for all smooth real valued harmonic functions $\phi$ which is Morse up to the boundary and has no critical point on the boundary.
\end{coro}
%\begin{proof}
%Suppose 
%\[\int_{\partial M_0} f i_{\partial M_0}^*\partial \phi = 0\]
%for all smooth real valued harmonic functions $\phi$ which is Morse up to the boundary. By density 
%this is satisfied for all $H^k(M_0)$ harmonic functions $\phi$ which is Morse up to the boundary (for any $k\in\nn$
%large enough so that $H^k(M_0)\subset C^{3}(M_0)$). 
%Since by  \cite[Lem. 2.2]{GT} the set of real valued harmonic Morse functions on $M_0$ is dense in the set of 
%real valued harmonic functions on $M_0$ for the $C^3$ topology, 
%we have that the above identity is satisfied by all real valued harmonic functions $\phi \in C^{3}(M_0)$.
%Furthermore, by linearity the identity holds for complex valued harmonic functions as well.
%Now let $\eta\in H^{3}(M_0 ; T^*_{1,0}M_0)$ satisfy $\bar\partial\eta = 0$. Then 
%by applying \cite[Th. C1.10]{McDSa} like in \cite[Cor. 2.3]{GT2} with a totally real 
%subbundle boundary condition having high boundary Maslov index, one obtains that the 
%operator $\partial : H^4(M_0)\to H^{3}(M_0; T^*_{1,0}M_0)$ is surjective, and  thus
%we have that $\eta = \partial \phi$ for some harmonic function $\phi\in H^{4}(M_0)$. 
%Therefore,
%\[\int_{\partial M} f i_{\partial M}^*\eta = \int_{\partial M} f i_{\partial M}^*\partial \phi = 0\]
%for all $\eta\in H^4(M ; T^*_{1,0}M)$ and the fact that $f$ is the restriction of a holomorphic function follows from %Lemma \ref{orthogonality bc}. The anti-holomorphic statement follows similarly and the converse is left again as an% exercise.
%\end{proof}

Let $\phi$ be a smooth real valued harmonic function defined on $M$. It is easy to see  that $\phi$ has a harmonic conjugate if and only if
\[M_j(\phi): = \int_{\gamma_j}*d\phi = 0 \ \ \ \ j = 1,.., n.\]
where $\{\gamma_1,..,\gamma_{n}\}$ is a family of generators of the fundamental group $\pi(M_0,m_0)$  
(with $m_0\in M_0$ fixed).  
However,  if $M(\phi) = (M_1(\phi),..,M_{n}(\phi))\in (2\pi\zz)^{n}$ then one can construct a multivalued function $\psi$ such that $F := e^{\phi+i\psi}$ is a (single-valued) holomorphic function with $|F| = e^\phi$ and $\partial F = (1-i)F\partial \phi$.
Indeed, it simply suffices to set $\psi(m)=-\int_{\gamma(m)} * d\phi$ where $\gamma(m)$ is a smooth path
joining $m_0$ to $m$.
%It turns out that in our case, it will be easier to verify the hypothesis of Corollary \ref{reduction} for real valued harmonic functions $\phi$ satisfying $M(\phi) \in (2\pi\qq)^{n}$.
\begin{lemma}
\label{more reduction}
Let $f\in H^{1/2}(\partial M_0)$ be a complex valued function. Then $f$ is the restriction of a holomorphic function if and only if
\begin{equation}\label{intbord}
\int_{\partial M_0} f i_{\partial M_0}^*\partial \phi = 0
\end{equation}
for all smooth real valued harmonic functions $\phi$ which is Morse up to the boundary and $M(\phi) \in (2\pi\zz)^{n}$. Similary, $f\in H^{1/2}(\partial M_0)$ is the restriction of an anti-holomorphic function if and only if
\[\int_{\partial M_0} f i_{\partial M_0}^*\bar{\pl}\phi = 0\]
for all smooth real valued harmonic functions $\phi$ which is Morse up to the boundary and $M(\phi) \in( 2\pi\zz)^{n}$.
\end{lemma}
\begin{proof}
By Corollary \ref{reduction}, we need to check that (\ref{intbord}) is equivalent to the condition stated in Corollary \ref{reduction}. Observe that if $M(\phi)\in(2\pi\qq)^n$, then $M(k\phi)\in(2\pi\zz)^n$ for some integer $k$, thus the condition
\eqref{intbord} is satisfied for all smooth harmonic Morse $\phi$ with $M(\phi)=(2\pi\qq)^n$ and no boundary critical point if and only if 
it is satisfied for all smooth harmonic Morse $\phi$ with $M(\phi)\in(2\pi\zz)^n$ and no boundary critical points. Therefore, 
by taking limits, it suffices to show that any Morse harmonic function $\phi$ with no critical point on the boundary can be approximated by a sequence of Morse harmonic functions with periods in $2\pi \qq$ in the $C^3(M_0)$ topology. 
Indeed, let $\{\phi_1,..,\phi_{n}\}$ be a set of real valued $C^3$ harmonic functions in $M_0$ such that $M_j(\phi_k) = \delta_{jk}$. Such harmonic functions exist due to equation (3.1) of \cite{CeFl} which states that the period matrix is invertible. Then given a real valued harmonic function $\phi$, we consider small perturbations $\phi' := \phi + \sum_{j = 1}^{n} \epsilon_j \phi_j$ of $\phi$. 
Since $\qq$ is dense in $\rr$, $\epsilon_j$ can be chosen arbitrarily small 
such that $M(\phi')\in (2\pi\qq)^{n}$. If $\phi$ is Morse with no critical points on the boundary and being Morse is an open condition, $\phi'$ is also Morse with no critical points on the boundary if $\epsilon_j$ are taken small enough.
%In \cite[Lem. 2.2.]{GT}, we proved that the real Banach space 
%$\mc{H}:=\ker_{C^3}(\Delta_g)\cap \ker( M)$ has the following property: the set of Morse functions in $\mc{H}$ is 
%dense for the $C^3$ topology. Since the proof of \cite[Lem. 2.2.]{GT} is based on the transversality 
%of the map $m: M_0\x \mc{H}\to TM_0,  m(z_0,u)=du(p)$ to the zero section of $TM_0$ and since $\mc{H}':=\phi+\mc{H}$ 
%and $\mc{H}$ have same tangent space, the same proof applies mutatis mutandis to the affine Banach manifold 
%$\mc{H}'$ and we obtain that the set of Morse functions in $\mc{H}'$ is dense for the $C^3$ topology.
%Therefore there exists $\phi''$ which is a $C^3$ approximation of $\phi$ with $M(\phi'')\in(2\pi\qq)^{n}$.
\end{proof}

\begin{lemma}
\label{generalized cgo}
Let $\phi$ be a harmonic Morse function such that $M(\phi) \in (2\pi\zz)^{n}$ and $\alpha_1,\alpha_2$ are functions 
such that $\bar\partial\alpha_1 = A_1$ and $\bar{\pl}\alpha_2 = A_2$. 
Then for all $k\in \zz$ large enough, there exists a solution $u_1\in H^1(M_0)$ to
\[L_1 u_1 = 0\]
such that 
\[u_1 = F^k (e^{-i\alpha_1} + r_1) \, \textrm{ with }\,  \sqrt{k}\|r_1\|_{L^2} +  \|r_1\|_{H^1}\leq C\] 
where $F$ is holomorphic,  
$\partial F = (1-i)F\pl\phi$ and  $|F| = e^{\phi}$. Similarly, there exists a solution $u_2\in H^1(M_0)$ to 
\[L_2^* u_2 = 0\]
such that
\[ u_2 = {\bar F}^{-k} (e^{-i\bar{\alpha}_2} + r_2)  \, \textrm{ with }\, \sqrt{k}\|r_2\|_{L^2} +  \|r_2\|_{H^1}\leq C. \]
\end{lemma}
\begin{proof}
Since $M(\phi)\in (2\pi\zz)^{n}$ we can construct a multivalued complex conjugate $\psi$ such that $F := e^{\phi + i\psi}$ is a well defined holomorphic function satisfying $\partial F = (1-i)F\partial\phi$ and $|F| = e^{\phi}$.
Since for any $\alpha_1$ satisfying $\partial \bar{\alpha}_1 = \bar{A}_1$, $L_1$ can be written as
\[L_1 = -2i*(\pl +i\bar{A}_1\wedge)(\bar\pl +iA_1) + Q_1 =
-2i* e^{-i\bar{\alpha}_1}\partial e^{i\bar{\alpha}_1}e^{-i\alpha_1}\bar\partial e^{i\alpha_1} + Q_1\]
for some $Q_1\in L^\infty(M_0)$, 
we have that $L_1 F^k e^{i\alpha_1} = F^ke^{i\alpha_1}Q_1$. 
By Corollary \ref{laxmilgram}, for all $|k|$ large enough there exists a $\tilde r_1$ solving
\[e^{-k\phi }L_1 e^{k\phi} \tilde r_1 = |F^{-k}| L_1 |F^k| \tilde r_1 = \tilde Q_1\]
such that
\[\sqrt{k}\|\tilde r_1\|_{L^2} + \|\tilde r_1\|_{H^1} \leq C\|\tilde Q_1\|_{L^2}.\]  
Setting $r_1 = \frac{F^k}{|F^k|} \tilde r_1$ we have that
\[L_1 F^k(e^{\alpha_1} + r_1) = 0\]
and that $r_1$ satisfies the desired estimates. 
The construction for $u_2$ follows similarly after factorizing  
\[L_2=-2i* e^{-i\alpha_1}\bar\partial e^{i\alpha_1}e^{-i\bar{\alpha}_2}\partial e^{i\bar{\alpha}_2} + \til{Q}_2\] 
for some $\til{Q}_2\in L^\infty(M_0)$ and using that $L_2^*-L_2$ is a zeroth order differential operator. 
\end{proof}

\noindent{\bf Proof of Proposition \ref{boundary of holomorphic}}: 
We need to show that if $\alpha_1$ and $\alpha_2$ are functions such that $\bar\partial\alpha_j = A_j$, 
then $e^{-i(\alpha_1 - \alpha_2)}\mid_{\pl M_0}$ is the boundary value of a holomorphic function. 
By Lemma \ref{more reduction} this is equivalent to showing that 
\[\int_{\pl M_0} f i_{\pl M_0}^*\partial \phi = 0\]
for all smooth real valued harmonic functions $\phi$ which is Morse up to the boundary and $M(\phi) \in (2\pi\zz)^{n}$.\\
Let $\phi$ be such a harmonic function and let $u_1 = F^k (e^{-i\alpha_1} + r_1)$ and $u_2 = {\bar F}^{-k} (e^{-i\bar{\alpha}_2} + r_2)$ be the solutions constructed in Lemma \ref{generalized cgo}. Plugging these solutions into the boundary integral identity 
\[\int_{M_0} \bar{u_2}(2(A_1 - A_2) \wedge \partial u_1 + 2(\bar{A}_1 - \bar{A}_2)\wedge \bar\partial u_1 + (Q_1 - Q_2)u_1) = 0\]
we have that 
\[2k(1-i) \int_{M_0} e^{-i(\alpha_1 -\alpha_2)}(A_1 - A_2)\wedge \partial \phi + O(\sqrt{k}) = 0.\]
Using the fact that $-ie^{-i(\alpha_1 -\alpha_2)}(A_1 - A_2) = \bar\partial e^{-i(\alpha_1 -\alpha_2)}$ we can integrate by parts and take $k\to\infty$ to get that
\[\int_{\pl M_0} e^{-i(\alpha_1 -\alpha_2)}i_{\partial M_0}^*\partial\phi = 0.\]
Since this is true for any real valued harmonic Morse function $\phi$ with $M(\phi) \in (2\pi\zz)^{n}$, $e^{-i(\alpha_1 -\alpha_2)}
|_{\pl M_0}$ is the boundary value of a holomorphic function by Lemma \ref{more reduction}.
 \qed\\

In view of this proposition we will denote by $F_{-i(\alpha_1-\alpha_2)}$ to be the unique holomorphic function with boundary 
value $e^{-i(\alpha_1-\alpha_2)}\mid_{\partial M_0}$. 
Observe that one can reverse the indices $1$ and $2$ in the proof above, and this shows that 
$e^{i(\alpha_1-\alpha_2)}|_{\pl M_0}$ is also the boundary value of a holomorphic function $F_{i(\alpha_1-\alpha_2)}$.
Therefore, writing $\alpha=\alpha_1-\alpha_2$, it is clear after remarking that the product
$F_{i\alpha}F_{-i\alpha}$  has boundary value $1$, that 
\begin{equation}\label{relationF}
F_{i\alpha}F_{-i\alpha}=1. 
\end{equation}

\end{section}

\begin{section}{Reduction to First Order Systems}
In this section we use the boundary identity result we obtained in Proposition \ref{boundary of holomorphic} 
to reduce the inverse problem for the second order equation to an inverse problem for the first order system of Dirac type on sections of 
the bundle $\Sigma=\Lambda^0\oplus \Lambda^{0,1}$ introduced in Section \ref{firstorder}. 
We will do this by factoring the operators $L_1$ and $L_2$ the appropriate way.\\

Let $\alpha_j$ satisfy $\bar\pl \alpha_j = A_j$ (this exists by Proposition \ref{cauchy riemann}), 
then we also have $\pl\bar{\alpha}_j=\bar{A}_j$. 
We set $\alpha:=\alpha_1-\alpha_2$, then by Proposition \ref{boundary of holomorphic} and \eqref{relationF}, 
we see that there exists a nonvanishing holomorphic function $F_{-i \alpha}$ such that 
$F_{-i\alpha}|_{\pl M_0} = e^{-i\alpha}|_{\pl M_0}$. In particular, if  
\begin{equation}\label{defFAj}
F_{A_2} := e^{i\alpha_2}\, \textrm{ and }\, F_{A_1} = F_{-i\alpha}e^{i\alpha_1}
\end{equation} 
then 
\begin{equation}
\label{boundary agree}
\bar\pl F_{A_j} = iA_j F_{A_j}\,\, \textrm{ and }\,\, F_{A_1}|_{\pl M_0} = F_{A_2}|_{\pl M_0}.
\end{equation}
Similarly, there is a unique nonvanishing anti-holomorphic function $\bar{F}_{i\alpha}$ 
such that $F_{\bar{A}_1} := \bar{F}_{i\alpha}e^{i\bar{\alpha}_1}$ and 
$F_{\bar{A}_2} := e^{i\bar{\alpha}_2}$ satisfy 
\begin{equation}
\label{boundary agree2}
\pl F_{\bar{A}_j} = iF_{\bar{A}_j}\bar{A}_j \,\, \textrm{ and }\,\, F_{\bar{A}_1}|_{\pl M_0} = F_{\bar{A}_2}|_{\pl M_0}.
\end{equation}
We can then write
\[ L_j = -2i* (\pl+i\bar{A}_j\wedge)(\bar\pl + iA_j) + Q_j = 
2 F_{\bar{A}_j}^{-1}\bar{\pl}^* F_{\bar{A}_j} F_{A_j}^{-1}\bar\pl F_{A_j} +  Q_j.\]
where $Q_j = * dX_j + q_j$. Let $u_j\in H^2(M_0)$ and set $\omega_j :=  F_{A_j}^{-1}\bar\pl F_{A_j} u_j\in H^{1}(M_0, T_{0,1}^*M_0)$, 
\[L_ju_j=0 \iff \begin{pmatrix}
0 & {\bar\pl}^*\cr
\bar{\pl} & 0
\end{pmatrix} 
\begin{pmatrix}
{F_{A_j}}&{0}\cr
{0}&{F_{\bar{A}_j}}
\end{pmatrix}
\begin{pmatrix}
{u_j}\cr
{\omega_j}
\end{pmatrix} + \begin{pmatrix}
F_{\bar{A}_j}Q_j/2 & 0\cr
0 & -F_{A_j}
\end{pmatrix} 
\begin{pmatrix}
{u_j}\cr
{\omega_j}
\end{pmatrix}=0.\]
Observe that  \eqref{relationF}  implies 
$\bar{F}_{\bar{A}_j} = F_{A_j}^{-1}$, therefore, 
if we set $(\tilde u_j, \tilde\omega_j) := (F_{A_j} u_j, F_{\bar{A}_j} \omega_j)=(F_{A_j}u_j,\bbar{F}_{A_j}^{-1}\omega_j)$, then $(u_j,\omega_j)$ 
solves the above system of equations if and only if $(\tilde u_j, \tilde\omega_j)$ solves
\begin{equation}\label{first order sys}
\begin{pmatrix}
0 & {\bar\pl}^*\cr
\bar{\pl} & 0
\end{pmatrix} 
\begin{pmatrix}
{\tilde u_j}\cr
{\tilde \omega_j}
\end{pmatrix} + \begin{pmatrix}
\demi Q_j|F_{A_j}|^{-2}& 0\cr
0 & -|F_{A_j}|^2
\end{pmatrix} 
\begin{pmatrix}
{\tilde u_j}\cr
{\tilde \omega_j}
\end{pmatrix}=0\end{equation}
Denoting the Cauchy data space for $L_j$ to be 
\[\mc{C}_{L_j}\:= \{(u, \nabla_\nu^{X_j} u)|_{\pl M_0}\in H^{\demi}(\pl M_0)\x H^{-\demi}(\pl M_0); u\in H^1(M_0), L_ju=0 \}\] 
where $\nabla^{X_j}_{\nu}u:=du(\nu)+iX_j(\nu)u$ and $\nu$ the unit normal interior vector field to $\pl M_0$,
we deduce from this discussion the 
\begin{proposition}
\label{cauchy data}
Assume that $X_1,X_2\in W^{2,p}(M_0,T^*M_0)$ are real valued and $q_1,q_2\in W^{1,p}(M_0)$ complex valued 
for some $p>2$. 
If $L_1$ and $L_2$ have the same Cauchy data space $\mc{C}_{L_1}=\mc{C}_{L_2}$, 
then the first order system \eqref{first order sys} with diagonal endomorphism $(\demi Q_1|F_{A_1}|^{-2}, |F_{A_1}|^2)$ 
of $\Sigma$ has the same Cauchy data space as the one with endomorphism $(\demi Q_2|F_{A_2}|^{-2}, |F_{A_2}|^2)$ 
for $Q_j=-dX_j+q_j$ and $F_{A_j}$ defined as above.
\end{proposition}
\begin{proof}
By the boundary determination and the remark following Proposition \ref{boundary determination}, we can suppose
that $X_1=X_2$ on $\pl M_0$ as forms, and therefore $A_1=A_2$ as well on $\pl M_0$.
Combining the discussion above with equalities \eqref{boundary agree}, \eqref{boundary agree2}, 
we have that $2$ solutions $u_j$ of $L_ju_j=0$ satisfying $(u_1-u_2)|_{\pl M_0}=0$
 are such that $ (\til{u}_1-\til{u}_2)|_{\pl M_0}=0$ and 
\[ i_{\pl M_0}^*(\til{\omega}_1-\til{\omega}_2)= 
i^*_{\pl M_0}\Big(\bar{F}_{A_1}^{-1}[(\bar{\pl}+iA_1)u_1-(\bar\pl+iA_2)u_2]\Big)= i_{\pl M_0}^*
(e^{i\bar{\alpha}_2} \bar{\pl}(u_1-u_2))\]
where $(\til{u}_j,\til{\omega}_j):=(F_{A_j}u_j,|F_{A_j}|^{-2}\bar\pl F_{A_j} u_j)$. Now,  
\[ \nabla^{X_j}_\nu u_j|_{\pl M_0}=(du_j(\nu)+iX_j(\nu)u_j)|_{\pl M_0}\]
and since $X_1=X_2$ on $\pl M_0$,  $\mc{C}_{L_1}=\mc{C}_{L_2}$ implies that
$d(u_1-u_2)(\nu)|_{\pl M_0}=0$, which together with $(u_1-u_2)|_{\pl M_0}=0$ gives $\til{\omega}_1=\til{\omega}_2$ on $\pl M_0$.
This achieves the proof.
\end{proof}

We deduce from Proposition \ref{cauchy data} and Theorem \ref{inverseresult} the following 
\begin{coro}\label{dX=0}
Let $X_j\in W^{2,p}(M_0,T^*M_0)$ be real valued and $q_j\in W^{1,p}(M_0)$ complex valued, and $p>2$. If 
the Cauchy data spaces for $L_1,L_2$ satisfy $\mc{C}_{L_1}=\mc{C}_{L_2}$,
 then $d(X_1-X_2)=0$ and $q_1=q_2$. 
\end{coro}
\begin{proof} Acoording to Proposition \ref{cauchy data} and Theorem \ref{inverseresult}, we have that
\begin{equation}\label{firstident}
-dX_1+q_1=-dX_2+q_2 \textrm{ and } |F_{A_2}|^2=|F_{A_1}|^2 
\end{equation}
where $X_j=A_j+\bar{A}_j$ and $A_j\in \Lambda^{0,1}, \bar{A}_j\in \Lambda^{1,0}$. The functions 
$F_{A_j}$ are defined  in \eqref{defFAj} and $\alpha_j$ satisfies $\bar{\pl}\alpha_j=A_j$. 
Since $|F_{A_2}|=e^{-{\rm Im}(\alpha_2)}=|F_{A_1}|=e^{-{\rm Im}(\alpha_1)}|F_{-i\alpha}|$
for $\alpha=\alpha_1-\alpha_2$ and $F_{i\alpha}$ is a holomorphic function which does not vanish 
in $M_0$ in view of \eqref{relationF} (the log of its modulus is then harmonic), then setting $A=\bar{\pl}\alpha=A_1-A_2$ we deduce  
that 
\[0=2i\Delta {\rm Im}(\alpha)=\Delta(\alpha-\bar{\alpha})=-2i* \pl \bar{\pl}\alpha-2i*\bar{\pl}\pl\bar{\alpha}=
-2i*(\pl A+\bar{\pl}\bar{A})=-2i* d(X_1-X_2)\]
and therefore $q_1=q_2$ as well by \eqref{firstident}.
\end{proof}
\end{section}

\begin{section}{Cauchy data determine the holonomy}\label{holonomy}

For each $m\in M_0$ and each closed loop $\gamma$ based at $m_0$,
 the parallel transport for the connection $\nabla^X$ on the bundle $M_0\x \cc$ defines
an isomorphism $P^X_\gamma: \cc\to \cc$ of the fiber $\cc$ at $m$, thus $P_\gamma$ can be viewed as a 
non-zero complex number $P^X_\gamma\in \cc\setminus \{0\}$. The holonomy group of $\nabla^X$ at $m$
is given by 
\[H_{m}(\nabla^X):=\{P^X_\gamma\in \cc\setminus\{0\}; \gamma \textrm{ is a closed loop based at }m \}.\]
For the connection $\nabla^X=d+iX$, an easy computation shows that 
\[P^X_\gamma = e^{-i\int_{\gamma}X}\]
where $\gamma$ is an oriented closed curve. 
Notice that if $X$ is real valued, the holonomy group $H_m(\nabla^X)$ is a subgroup of $S^1$.
If $X$ is a flat connections $1$-forms, ie. with cuvature $dX=0$, 
then the map $\gamma\to P^X_\gamma$ induces a natural group morphism $\rho_m^X:\pi_1(M_0,m)\to H_m(\nabla^X)$
where $\pi_1(M_0,m)$ is the fundamental group based at $m$, ie. the set of closed 
loop up to homotopy equivalence. The morphism $\rho^X$ is called the holonomy representation into $GL(\cc)$ 
and it is trivial if and only if 
\[ e^{-i\int_{\gamma}X}=1 \textrm{ for all closed loop }\gamma \textrm{ based at }m,\]
this condition is also independent of $m$. If $X_1,X_2$ are two connection $1$-forms with same curvature $dX_1=dX_2$, and if 
the holonomy representation of $X:=X_2-X_1$ is trivial, then there exist a unitary 
bundle isomorphism $F:E\to E$ (recall that $E=M_0\x \cc$), 
or equivalently a function $F:M_0\to \cc$ of modulus $|F|=1$, defined by 
\[F(m')=e^{i\int_{\gamma(m,m')} X}\]
where $\gamma(m,m')$ is any $C^1$ path joining $m$ and $m'$, this is well defined independently of the path
since $dX=0$ and thanks to the triviality of the holonomy representation. The connections $X_1$ and $X_2$ are related by
$F^*(d+iX_1)F=d+iX_2$, and if moreover $i^*_{\pl M_0}X_1=i_{\pl M_0}^*X_2$, then the isomorphism $F$ is the identity
when restricted to $\pl M_0$. 

In view of this discussion, to prove Theorem \ref{main}, we need to prove
\begin{theorem} \label{holo}
Let $X_1,X_2\in W^{2,p}(M_0)$ and $q_1,q_2\in W^{1,p}(M_0)$ for some $p>2$. Then the 
Cauchy data spaces $\mc{C}_{L_1}$ and $\mc{C}_{L_2}$ coincide if and only if 
$q_1=q_2$, $\nabla^{X_1}$ and $\nabla^{X_2}$ have same curvature $dX_1=dX_2$, and  
the holonomy representation $\rho_m^X$ is trivial for each $m\in M_0$, where we have set $X:=X_1-X_2$.
\end{theorem}
\begin{proof} We shall give two different proofs, the first one using directly the Cauchy data space, the other one using Proposition 
\ref{cauchy data}.\\
 
\textbf{First Proof}. We have already shown that $d(X_1 - X_2) = 0$ and $q_1 = q_2 $. 
Furthermore, by boundary determination (Proposition \ref{boundary determination} and the remark that follows),
 we can conclude that the tangential components 
of $X_1$ and $X_2$ agree along the boundary, ie. $i_{\pl M_0}^*(X_1-X_2)=0$, and 
that there exist a function $\zeta$ vanishing on the boundary such that $X_1+d\zeta=X_2$ on $\pl M_0$. 
Since the addition of an exact form as above does not change the Cauchy data space, we may assume
without loss of generality that $X_1=X_2$ at $\pl M_0$.
Let $[\gamma]\in \pi_1(M_0,p_1)$ be an equivalence class of loop and $\gamma$ be a representative. 
Since $X:=X_1 - X_2$ is closed, $\int_{\gamma}X$ is independent of the chosen representative. 
We choose a simple (non self-intersecting) representative $\gamma$ based at $p_1$, made of $2$ oriented pieces $[p_1, p_2]$ 
and $[p_2,p_1]$ such that $[p_2,p_1]\subset \pl M_0$.
This is possible for each primitive class (or generator) of  $[\gamma]\in \pi_1(M_0,p_1)$, ie. every class which can not be expressed 
as a power of another class. All the other classes in $\pi_1(M_0,p_1)$ 
are obtained by products of primitive classes (in the group law) and therefore the integrals
of $X$ along these classes are obtained by linear combinations over $\zz$ of integrals of $X$ on primitive classes, therefore it suffices to compute the integrals on $X$ on simple representatives (primitive classes) in $\pi_1(M_0,p_1)$. To prove the statement about the trivial holonomy representation, it suffices to prove that $\int_\gamma X\in 2\pi \zz$.

Let $\gamma: [0,2] \to M$ be a parametrization of this loop in such a way that $\gamma(0) = p_1$, $\gamma(1) = p_2$, 
$\gamma_1 := \gamma((0,1)) \subset {\rm int}(M_0)$, and $\gamma_2:=\gamma([1,2])\subset \partial M_0$. 
Now consider a thin tubular neighbourhood $\mc{O} := \{p\in {\rm int}M_0; {\rm dist}(p, \gamma_1) <\epsilon\}$ 
which is homeomorphic to $(0,1)\times (-\epsilon,\epsilon)$ and therefore simply connected, in particular we need 
to take $\eps$ so that ${\rm dist}(p_1,p_2)>2\eps$.
We define for points $p\in \bar{\mc{O}}$ the function $\alpha(p) := \int_{p_1}^pX$, so that  $d\alpha=X$
in $\bar{\mc{O}}$ (since $X$ is closed). Since $X=0$ on $\pl M_0$, we also have $\alpha|_{U_1} = 0$ where 
$U_1=\{p\in \pl M_0; {\rm dist}(p,p_1)\leq \eps\}$.
Now let $f$ be a smooth function defined on the boundary such that $f(p_1) = f(p_2) = 1$ and let $u_j$ solve (for $j=1,2$)
\[L_j u_j = 0,\ \ \  u|_{\pl M_0}=f.\]
Let $u := e^{i\alpha}u_1$ be defined in the tubular neighbourhood $\mc{O}$ and 
we shall now show that $u = u_2$ in $\mc{O}$ by unique continuation if the Cauchy data spaces $\mc{C}_{L_j}$ agree. 
Indeed, using that $\nabla^{X_2}e^{i\alpha}=e^{i\alpha}\nabla^{X_1}$ in $\mc{O}$ by simple calculation,  we deduce that 
$u$ solves $L_2u = 0$ in $\mc{O}$ since  
\[e^{-i\alpha}({\nabla^{X_2}}^*\nabla^{X_2}+q_2)e^{i\alpha}u_1=({\nabla^{X_1}}^*\nabla^{X_1}+q_1)u_1=0.\]
Furthermore, since $\alpha(p) = 0$ for all $p\in U_1$, one has $u(p) = u_1(p) = f(p) = u_2(p)$ for all $p\in U_1$.
 Moreover, since the normal components of $X_1$ and $X_2$ agree on $U_1$ as well, 
 the normal derivatives  satisfy
\[\partial_\nu u|_{U_1} = \pl_\nu u_1|_{U_1}  = \nabla_\nu^{X_1}u|_{U_1}-iX_1(\nu)f
=\nabla_\nu^{X_2}u|_{U_1}-iX_2(\nu)f=\partial_{\nu} u_2|_{U_1}\]
where in the third  equality we used the fact that the Cauchy data spaces $\mc{C}_{L_1}$ and $\mc{C}_{L_2}$ 
agree. So we have that $(u- u_2)$ is the solution of a homogenous elliptic equation in $\mc{O}$, vanishing on $U_1$ and 
with normal derivative vanishing on $U_1$, therefore by standard unique continuation 
\[e^{i\alpha(p)}u_1(p) = u(p) = u_2(p), \,\,\forall p\in \mc{O}.\] 
Now letting $p$ converging to $p_2 \in \bar{\mc{O}}$ and using the fact that $u_1|_{\pl M_0} = 
u_2|_{\pl M_0} = f$, 
we have that $e^{i\alpha(p_2)}f(p_2) = f(p_2)$. And since $f(p_2) = 1$ by assumption, 
we deduce that $e^{i\alpha(p_2)}= 1$ and consequently
\[\alpha(p_2) = \int_{\gamma_1}X \in 2\pi \zz.\]
Now since the tangential component of $X$ vanishes 
along the boundary and $\gamma_2\subset \pl M_0$, we deduce that $ \int_{\gamma_2}X = 0$ and
 conclude that
\[ \int_{\gamma}X =  \int_{\gamma_1}X \in 2\pi\zz\]
and the proof is complete.\\

\textbf{Second Proof}. Consider the functions $F_{A_j}$ of \eqref{defFAj}, then by Proposition \ref{cauchy data} and using 
$\bar{F}_{\bar{A}_j}=F_{A_j}^{-1}$, we know that 
$\Theta:=F_{A_1}/F_{A_2}=\bar{F}_{\bar{A}_2}/\bar{F}_{\bar{A}_1}$ is a function mapping $M_0$ to the unit circle $S^1\subset \cc$, 
and by \eqref{boundary agree}, \eqref{boundary agree2} we also have 
\[\bar{\pl}\Theta/\Theta= i(A_1-A_2), \quad \pl\Theta/\Theta=i(\bar{A}_1-\bar{A}_2)\]
and thus $d\Theta/\Theta=i(X_1-X_2)$. 
Let $\gamma,\mc{O}, p_1,p_2$ be like in the First Proof just above, 
we want to prove that $\int_{\gamma}d\Theta/\Theta\in 2i\pi\zz$.
In the tubular neighbourhood $\mc{O}$ of $\gamma$, we define 
$G(p):= i\int_{p_1}^p (X_1-X_2)=\int_{p_1}^p d\Theta/\Theta$, which is well defined in the simply connected domain $\mc{O}$ since 
$d(X_1-X_2)=0$, and $dG=d\Theta/\Theta$ in $\mc{O}$ with $G(p_1)=0$. The function $e^{G}$ takes value in $S^1$ and 
 we have then proved that $e^G=\Theta$ since $\Theta(p_1)=1$. Then, to conclude, it suffices to notice that $\Theta(p_2)=1$ 
and so $G(p_2)\in 2\pi i\zz$.
\end{proof}

\end{section}
\section{Appendix}
In this appendix, we gather a couple of technical results which are essentially already proved in the literature.
First, we give a Carleman estimate 
\begin{lemma} \label{CE}
Let $X\in W^{2,\infty}(M_0, T^1M_0)$ be real valued and $q\in W^{1,\infty}(M_0)$ complex valued and  set
$L={\nabla^{X}}^*\nabla^X+q$. Let $\phi$ be some harmonic real valued Morse function. Then there exists $C>0$ such that for all large $k\in \nn$ and all $u\in H^2(M_0)$
\[|| e^{-\phi/h} L e^{\phi/h}u||^2_{L^2}\geq C(\frac{1}{h}||u||^2_{L^2}+  ||du||^2_{L^2}).\]
\end{lemma}
\begin{proof} We observe that $L$ is a first order perturbation of the Laplacian $\Delta_g$ on $M_0$, therefore the 
Carleman estimate obtained for $\Delta_g$ in Lemma 3.2 of \cite{GT} 
with a convexified weight $\phi_\eps$ allows to absorb the first order terms by taking $\eps>0$ 
small and this shows the result for $L$ for the convexified weight. Then the argument of Proposition 3.1 in \cite{GT} shows 
that the desired estimate for $L$ holds for the weight $\phi$.
\end{proof}
As a corollary
\begin{coro}\label{laxmilgram}
With the same assumptions as in Lemma \ref{CE}, there exists $h_0>0$ and $C>0$ 
such that for all $h\in(0,h_0)$ and for all $f\in L^2(M_0)$ 
there exist a solution $u\in H^2(M_0)$ of $e^{-\phi/h}Le^{\phi/h}u=f$ with 
norms $||u||_{L^2}\leq C\sqrt{h}||f||_{L^2}$ and $||du||_{L^2}\leq C||f||_{L^2}$.
\end{coro}
\begin{proof} The proof is a standard application of Lax-Milgram theorem (or Riesz representation) with the estimate of Lemma \ref{CE}, in exactly the same way as Lemma 4.4 of \cite{GT}. 
\end{proof}

The boundary determination is standard, but since there seem to be no proof in the case of the system $D+V$ studied in section \ref{firstorder}, we provide
a sketch of proof, based essentially on the arguments of \cite[Appendix]{GT2}. 
\begin{lemma}\label{boundarydet}
With the notations of Section \ref{firstorder}, let $V_1,V_2\in W^{1,p}(M_0,{\rm End}(\Sigma))$ be two diagonal 
complex valued potentials endomorphisms of $\Sigma$. Assume that the Cauchy data spaces $\mc{C}_{V_1}$ 
of $D+V_1$ and $\mc{C}_{V_2}$ of $D+V_2$ agree, then $i_{\pl M_0}^*V_1=i_{\pl M_0}^*V_2$.
\end{lemma}
\begin{proof}
Let $H^1_0(M_0,\Sigma)$ be the completion of $C_0^\infty({\rm int}(M_0))$
for the $H^1(M_0,\Sigma)$ topology, and $H^{-1}(M_0,\Sigma)$ the dual space.
By standard arguments (for instance Carleman estimates and Lax-Milgram theorem), we have that 
for all $W\in L^2(M_0,\Sigma)$, there exists a $U\in H^1(M_0,\Sigma)$ such that  
$(D+V_j)U=W$ and $||U||_{L^2}\leq C||W||_{H^{-1}}$. Now let 
$A_h=(a_h,0)\in C^{\infty}(M_0,\Sigma)$ with $a_h$ a function supported in a chart $\mc{U}_p$
near a boundary point $p$ and defined as follows: if $z=x+iy$ are complex coordinates near $p$ with 
$\{y=0\}=\pl M_0\cap \mc{U}_p$, $M_0\cap \mc{U}_p=\{y\geq 0\}$ 
and $p=\{z=0\}$ in this chart, 
then we set $a_h(z):=\eta(zh^{-\alpha})e^{iz/h}$ 
with $\eta\in C_0^\infty(\cc)$ supported in the chart and equal to $1$ at $z=0$, and $\alpha\in(0,1/2)$. 
Notice that $||A_h||_{L^2}=O(h^{\demi (1 +\alpha)})$ and 
 $DA_h=(0,h^{-\alpha} \pl_{\bar{z}}\eta (zh^{-\alpha})e^{iz/h})$ has $L^2$ norm $O(h^{\demi(1-\alpha)})$.
Let $=(f_1,f_2d\bar{z})\in H^1_{0}(M_0,\Sigma)$, then one has (the metric $g$ is of the form $e^{2\rho}|dz|^2$ for some smooth function 
$\rho$)
\[\cjg DA_h,F\cjd=-ih^{-\alpha+1}\int_{\mc{U}_p} \pl_z(e^{iz/h}) \cjg \pl_{\bar{z}}\eta (zh^{-\alpha})d\bar{z},f_2\cjd e^{2\rho}dxdy\] 
and integrating by parts, we loose at most a power $h^{-\alpha}$ when the derivative hits $\eta(zh^{-\alpha})$. Using 
Cauchy-Schwartz, we deduce $|\cjg DA_h,F\cjd|\leq Ch^{\frac{3}{2}(1-\alpha)}||F||_{H^1}$ and therefore 
$||DA_h||_{H^{-1}}=O(h^{\frac{3}{2}(1-\alpha)})$. Adding a potential is harmless and thus 
$||(D+V)A_h||_{H^{-1}}=O(h^{\frac{3}{2}(1-\alpha)})$.
Taking $\alpha=\frac{1}{3}$ for instance, we obtain that there exists $Z_h\in H^1(M_0,\Sigma)$, with norm 
$||Z_h||_{L^2}=O(h)$ such that $(D+V)F_h=0$ with $F_h=A_h+Z_h$. 
We get these solutions $F_h^1,F_h^2$ for the diagonal potentials $V_1=\begin{pmatrix}{v_1} & {0}\cr {0} & {v_1'}\end{pmatrix}$ 
and $V_1=\begin{pmatrix}{v_2} & {0}\cr {0} & {v_2'}\end{pmatrix}$ and plug them 
into the integral identity \eqref{integid}, giving then by elementary computations (and using $v_j,v_j'\in W^{1,p}(M_0)$)
 \[0=\int_{M_0} (v_1-v_2)\eta^{2}(zh^{-\frac{1}{3}})e^{-2y/h}e^{2\rho}dxdy+o(h^{\frac{4}{3}})=
 C(v_1(p)-v_2(p))h^{\frac{4}{3}}+o(h^{\frac{4}{3}}).\]
for some $C\not=0$. This proves that $v_1=v_2$ at $p$, the same argument can be used to prove that $v_1'=v_2'$ at $p$
and since $p$ is arbitrarily chosen, we have achieved the proof of the Lemma. 
\end{proof}

\end{document}